\documentclass[11pt]{amsart}

\title[Automorphisms of $S^3$ that preserve spatial graphs and handlebodies]
{Automorphisms of the 3-sphere that preserve spatial graphs and handlebody-knots}
\author{Yuya Koda}
\thanks{The author is supported by JSPS Postdoctoral Fellowships for Research Abroad, and
by the Grant-in-Aid for Young Scientists (B), JSPS KAKENHI Grant Numbers 20525167 and 26800028.}
\address{
Department of Mathematics \newline 
\indent Hiroshima University, 1-3-1 Kagamiyama, Higashi-Hiroshima, 739-8526, Japan}
\email{ykoda@hiroshima-u.ac.jp}

\usepackage{amsfonts,amsmath,amssymb,amscd}

\usepackage{amsthm}
\usepackage{latexsym}
\usepackage[dvips]{graphicx}
\usepackage[dvips]{psfrag}

\theoremstyle{plain}
\newtheorem*{theorem*}{Theorem}
\newtheorem*{lemma*} {Lemma}
\newtheorem*{corollary*} {Corollary}
\newtheorem*{proposition*}{Proposition}
\newtheorem*{conjecture*}{Conjecture}
\newtheorem{theorem}{Theorem}[section]
\newtheorem{lemma}[theorem]{Lemma}
\newtheorem{corollary}[theorem]{Corollary}
\newtheorem{proposition}[theorem]{Proposition}

\theoremstyle{remark}

\newtheorem*{notation}{Notation}
\newtheorem*{example}{Example}

\theoremstyle{definition}

\newcommand{\Integer}{\mathbb{Z}}

\newcommand{\Int}{\mathrm{Int}}

\newcommand{\MCG}{\mathcal{MCG}}

\newcommand{\Aut}{\mathrm{Aut}}
\newcommand{\Sym}{\mathrm{Sym}}
\newcommand{\TSG}{\mathrm{TSG}}
\newcommand{\Stab}{\mathrm{Stab}}
\newcommand{\rel}{\mathrm{rel}}
\newcommand{\id}{\mathrm{id}}



\begin{document}
\maketitle

\begin{abstract}
We consider the group of isotopy classes of automorphisms of 
the 3-sphere that preserve a spatial graph or a handlebody-knot 
embedded in it. 
We prove that the group is finitely presented for 
an arbitrary spatial graph or a reducible handlebody-knot 
of genus two. 
We also prove that the groups for ``most" irreducible 
genus two handlebody-knots are finite. 
\end{abstract}

\vspace{1em}

\begin{small}
\hspace{2em}  \textbf{2010 Mathematics Subject Classification}: 
57M15; 20F05, 57M25. 

\hspace{2em} \textbf{Keywords}: spatial graph, handlebody, mapping class group.
\end{small}

\section*{Introduction}

In the present paper, we consider two natural 
generalizations of classical knot theory: the studies of spatial 
graphs and handlebody-knots.

By a {\it graph} 
we shall mean the underlying 
space of a finite simplicial complex of dimension 
at most one. 
A {\it spatial graph} is a graph embedded 
in a $3$-manifold. 
Two spatial graphs are said to be equivalent 
if one can be transformed into the other by
an ambient isotopy of the $3$-manifold. 
Since a circle is a graph, the study of 
spatial {graphs} is 
a generalization of classical knot theory. 
As in classical knot theory, 
we can study spatial graphs in combinatorial, topological, algebraic  
and geometric ways. 
For example, we can introduce hyperbolic geometry 
into the study of spatial graphs. 
In fact, most pairs of 3-manifolds and  
spatial trivalent graphs in them 
admit unique hyperbolic structures with parabolic 
meridians, that is, 
the complements of the spatial graphs 
carry (possibly-incomplete) metrics of constant sectional curvature 
$-1$ which complete to 3-manifolds with non-compact 
geodesic boundary having: toric cusps at 
the knot components of the graph,  
annular cusps at the meridians of the edges, and 
geodesic 3-punctured boundary spheres at the vertices, 
see e.g. Heard-Hodgson-Martelli-Petronio \cite{HHMP10}. 
In this way, one can define hyperbolic invariants, such as volumes and arithmetic invariants, 
for most spatial trivalent graphs. 
We also note that the above hyperbolic structure are useful 
when we study orbifold structures, where 
spatial trivalent graphs appear 
as the singular loci of 3-orbifolds, see e.g. Cooper-Hodgson-Kerckhoff \cite{CHK00}.

A {\it handlebody} is a compact orientable 3-manifold obtained by 
attaching a finite number of 1-handles to the 3-ball. 
The nuber of the attached 1-handles is called the {\it genus} of 
the handlebody. 
A {\it handlebody-knot} is a handlebody embedded in a 3-manifold. 
Two handlebody-knots are said to be equivalent 
if one can be transformed into the other by
an ambient isotopy of the 3-manifold. 
We note that the study of genus 1 handlebody-knots is 
precisely the theory of classical knots. 
Handlebody-knots naturally appear 
in the theory of 3-manifolds as well, 
since every closed orientable 3-manifold 
admits a decomposition into two handlebodies called 
a Heegaard splitting and these handlebodies can be regarded as 
the simplest handlebody-knots in the 3-manifold. 
For studying handlebody-knots, we can again use hyperbolic geometry.  
In fact, the exteriors of most handlebody-knots 
admits hyperbolic structures with geodesic boundary. 
We note that handlebody-knots appear in the study of finiteness properties 
of invariants of 3-manifolds as well. 

Spatial graphs are more rigid objects than handlebody-knots in the sense that 
a spatial graph defines a unique handlebody-knot by taking its regular neighborhood 
up to equivalence, 
whereas regular neighborhoods of 
many mutually inequivalent spatial graphs may define equivalent handlebody-knots.  
For example, a regular neighborhood of the 
union of a tunnel number one knot in $S^3$ and 
its unknotting tunnel is always equivalent to 
the trivial handlebody-knot of genus two, that is, 
the genus two handlebody-knot whose exterior is also a handlebody. 
For the case of trivalent graphs, 
it is known that two spatial trivalent graphs 
define equivalent handlebody-knots 
if and only if one can be transformed 
into the other by
a finite sequence of IH-moves 
(or Whitehead moves) on the spatial graphs and by ambient 
isotopies of the 3-manifold, see Ishii \cite{Ish08}. 
Here an IH-move on a spatial graph 
is a replacement of 
a subspace of the spatial graph shaped like the letter $I$ in the  
spatial graph by a space shaped like the letter $H$.  

\vspace{1em}

For a subspace $N$ of a compact orientable 3-manifold $M$, 
the {\it mapping class group} of the pair $(M , N)$ 
is defined to be the group of isotopy classes of 
the automorphisms of $M$ that preserve $N$. 
We denote this group by $\MCG(M, N)$ and call it the 
{\it symmetry group} of $(M, N)$ as well to emphasize that it 
describes the symmetry of the subspace $N$ in $M$. 
The aim of the present paper is to investigate whether  
the symmetry group $\MCG(M, N)$ is 
finitely presented when 
$M = S^3$ and $N$ is a spatial graph or a 
handlebody-knot in $S^3$. 

Recall the case where 
$M$ is the 3-sphere $S^3$ and $N = K$ is a knot in $S^3$. 
Then the symmetry group $\MCG (S^3 , K) $  
has been widely studied by 
various authors, see e.g. Boileau-Zimmermann \cite{BZ89} and Kawauchi \cite{Kaw96}. 
In this case, the group $\MCG (S^3 , K)$ is isomorphic to 
the mapping class group of 
the exterior of $K$ due to 
a celebrated result of Gordon-Luecke \cite{GL89}, 
which states that if two knots in $S^3$
have homeomorphic complements then the homeomorphism between the
two complements extends to an automorphism of $S^3$. 
It follows that the symmetry group 
$\MCG (S^3 , K)$ is finitely presented from 
the result of McCullough \cite{McC91} on 
the mapping class groups of irreducible 
orientable sufficiently large 
$3$-manifolds. 
Note, in contrast, that it is
known that handlebody-knots of arbitrary genera 
are not determined 
by their complements, see Motto \cite{Mot90} and Lee-Lee \cite{LL12}. 

When $V$ is a genus $g$ unknotted handlebody 
in the 3-sphere $S^3$, the group $\MCG_+(S^3, V)$ is called 
the {\it Goeritz group} of genus $g$. 
In the case of genus two, 
it was shown by Goeritz \cite{Goe33} in 1933 
that $\MCG_+(S^3, V)$ is finitely generated. 
Scharlemann \cite{Sch04} gave a modern, 
concise proof of this result, 
and Akbas \cite{Akb08} and Cho \cite{Cho08} independently 
gave a finite presentation of $\MCG_+(S^3, V)$. 
However, in the case of higher genus, 
the existence of such generating sets still remains unknown. 
The best general reference here is the introduction of \cite{Cho08}. 
It is worth pointing out that finding a generating set for the 
symmetry group of an unknotted handlebody is 
very important for both the theory of handlebody-knots 
and that of Heegaard splittings.

\vspace{1em}

The paper is organized as follows. 
In Section \ref{sec:Preliminaries}, we set up notation and terminology. 
Then we review some of the standard facts about 
symmetry groups and the history of studying them. 
In Section \ref{sec:Symmetry groups of spatial graphs}, we discuss the case where 
$M$ is the 3-sphere $S^3$ and $N = \Gamma$ is a spatial graph in $S^3$. 
Our first main result, Theorem \ref{thm: symmetry group of a graph is finitely presented}, 
states that the symmetry group  $\MCG (S^3 , \Gamma) $ 
is finitely presented for an arbitrary spatial graph $\Gamma$ in $S^3$. 
This is a generalization of a known result for knots to spatial graphs.  
From Section \ref{sec:The symmetry groups of reducible handlebody-knots of genus two} on, 
we consider the symmetry groups of handlebody-knots. 
We note that 
we have $\MCG (S^3, K) \cong \MCG (S^3 , N(K))$ 
for a knot $K \subset S^3$ by the uniqueness of its regular neighborhood $N(K)$. 
This implies that the theory of the symmetry groups of handlebody-knots 
is another generalization of that of knots. 
In Section \ref{sec:The symmetry groups of reducible handlebody-knots of genus two}, 
we consider the group $\MCG(S^3, V)$,  
where $V$ is a reducible handlebody-knot of genus two. 
We define the complex $\mathcal{P}(V)$ of primitive disks, 
which is a subcomplex of the disk
complex of a handlebody $V$. 
The construction is similar in spirit to the primitive disk complex for the 
genus two handlebody standardly embedded in $S^3$ defined by Cho \cite{Cho08}. 
We prove that the complex
$\mathcal{P}(V)$ is contractible and 1-dimensional, i.e. a tree, 
on which $\MCG(S^3, V)$ acts simplicially. 
We then conclude in Theorem 
\ref{thm:symmetry group of reducible handlebody-knots of genus two} 
that the symmetry group of a reducible handlebody-knot of genus two 
is finitely presented, using the Bass-Serre theory \cite{Ser77} of groups acting on graphs.  
It is worth noting that 
Cho \cite{Cho13} recently proved that the primitive disk complex of 
the unique genus two 
Heegaard splitting of a lens space is 
1-dimensional or 2-dimensional, 
depending on the lens space. 
In Section \ref{sec:The symmetry groups of irreducible handlebody-knots of genus two}, 
we provide a sufficient condition for the symmetry group of 
an irreducible handlebody-knot of genus two 
to be finite. 
It is highly expected that 
the symmetry groups of all irreducible handlebody-knots of genus two 
are finitely presented. 
We prove that ``most" of them are, in fact, finite groups, however, 
the result in the section is far from being conclusive. 

\vspace{1em}

Throughout this paper, we will work in the piecewise linear category. 

\begin{notation}
Let $X$ be a subset of a given polyhedral space $Y$. 
Throughout the paper, we will denote the interior of 
$X$ by $\Int X$ 
and the number of components of $X$ by $\# X$. 
We will use $N(X; Y)$ to denote a closed regular neighborhood of $X$ in $Y$. 
If the ambient space $Y$ is clear from the context, 
we denote it briefly by $N(X)$. 
By a $3$-manifold we always mean one that is connected and compact, 
unless otherwise mentioned. 
Let $M$ be a 3-manifold. 
Let $L \subset M$ be a submanifold with or without boundary. 
When $L$ is 1 or 2-dimensional, we write 
$E(L) = M \setminus \Int N(L)$. 
When $L$ is of 3-dimension, we write 
$E(L) = M \setminus \Int L$. 
\end{notation}

\section{Preliminaries}
\label{sec:Preliminaries}

Let $N_1, N_2 , \ldots, N_n, L$ be possibly empty subspaces 
of a compact orientable 3-manifold $M$. 
We will denote by $\Aut (M, N_1, N_2 , \ldots, N_n ~\rel~ L )$ 
the group of 
automorphisms of $M$ which {map} $N_i$ onto 
$N_i$ for any $i=1 , 2 , \ldots , n$ and 
which {are} {the} identity on $L$. 
The {\it mapping class group}, denoted by 
$\MCG (M , N_1 , N_2 , \ldots, N_n ~\rel~ L )$, 
is defined to be the group of isotopy classes of 
elements of $\Aut (M , N_1 , N_2 , \ldots, N_n ~\rel~ L )$. 
When $L = \emptyset$, we will drop $\rel~L$. 
Given an element $f$ of the group of 
automorphisms $\Aut (M, N_1, N_2 , \ldots, N_n  ~\rel~ L )$, 
we use the same letter $f$ for 
its isotopy class by abuse of notation.  
%For simplicity of notation, we write 
%$\Aut(M)$ and $\MCG(M)$ instead of 
%$\Aut(M, \emptyset)$ and $\MCG(M, \emptyset)$ 
%respectively. 
The ``plus" subscripts, for instance in 
$\Aut_+ (M, N_1, N_2 , \ldots, N_n  ~\rel~ L )$ and 
$\MCG_+ (M, N_1, N_2 , \ldots, N_n ~\rel~ L )$, 
indicate the subgroups of 
$\Aut (M, N_1, N_2 , \ldots, N_n  ~\rel~ L )$ and 
$\MCG (M, N_1, N_2 , \ldots, N_n ~\rel~ L )$, 
respectively, consisting of 
orientation-preserving automorphisms 
(or their classes) of $M$. 	 
   
Let $N$ be a compact 3-manifold 
with boundary embedded in a compact orientable 3-manifold $M$. 
An essential disk $D$ in $N$ is said to be {\it primitive} 
if there exists an essential disk $E$ in $E(N)$ such that 
$\partial D$ and $\partial E$ have 
a single transverse intersection in $\partial N$. 
A pair of disjoint, non-isotopic primitive disks in $N$ 
is called a {\it primitive pair} in $N$.  

\subsection{Handlebody-knots}
\label{sec:Handlebody-knots}

We first review the following well-known fact about groups.  
\begin{lemma} 
\label{lem:exact sequence and finite presentation}
Let $1 \to G^\prime \to G \to G^{\prime \prime} \to 1$ 
be an exact sequence of groups.
If $G^\prime$ and $G^{\prime \prime}$ 
are finitely presented, so is $G$. 
In particular, for possibly empty subspaces 
$N_1, N_2 , \ldots, N_n$ 
of a compact orientable $3$-manifold $M$, 
the group $\MCG (M, N_1, N_2 , \ldots, N_n )$ 
is finitely presented if and only if 
so is $\MCG_+ (M, N_1, N_2 , \ldots, N_n )$. 
\end{lemma} 
\begin{proof}
The first assertion is well-known, and 
the particular case is also clear from the existence of 
the following exact sequence:  
\begin{eqnarray*}
1 &\to& \MCG_+ (M, N_1, N_2 , \ldots, N_n ) \to 
\MCG (M, N_1, N_2 , \ldots, N_n ) \\ 
&\to& 
\MCG (M, N_1, N_2 , \ldots, N_n ) / 
\MCG_+ (M, N_1, N_2 , \ldots, N_n ) \to 1 
\end{eqnarray*}
and the fact that subgroups of 
finite index in finitely presented groups are finitely presented, 
see e.g. Sims \cite{Sim94}.  
\end{proof}

%In this paper, we focus on handlebody-knots in the 
%3-sphere $S^3$. 
From now on, we will denote by $(S^3, V)$ 
a handlebody-knot $V$ in $S^3$. 
The {\it genus} of a handlebody-knot $(S^3, V)$ is 
defined to be the genus of the handlebody $V$. 
A handlebody-knot $(S^3, V)$ is said to be {\it trivial} 
if the exterior $E(V)$ is also a handlebody.  
It is equivalent to say that 
a handlebody-knot $(S^3, V)$ is trivial 
if and only if the pair 
$(V, E(V))$ determines a Heegaard splitting of 
$S^3$. 
By Waldhausen \cite{Wal68} there exists a unique 
trivial handlebody-knot of a given genus in $S^3$ up to isotopy. 
A 2-sphere $P$ in $S^3$ is called a 
{\it reducing sphere} for a 
handlebody-knot $(S^3, V)$ if 
$P$ intersects $V$ in a single essential disk. 
We note that $V \cap P$ and $E(V) \cap P$ are 
essential separating disks
in $V$ and $E(V)$ respectively. 
A handlebody-knot $(S^3, V)$ is said to be {\it reducible} 
if there exists a reducing sphere for $(S^3, V)$. 
Otherwise, $(S^3, V)$ is said to be {\it irreducible}. 
It is proved in Benedetti-Frigerio \cite{BF13} that a handlebody-knot $(S^3, V)$ 
of genus two is irreducible if and only if its exterior 
$E(V)$ is boundary-irreducible, i.e. 
$\partial E(V)$ is incompressible in $E(V)$. 
See e.g. Ishii \cite{Ish08}, Ishii-Kishimoto-Moriuchi-Suzuki \cite{IKMS12} and the references given there 
for more details. 

%Let $V \subset S^3$ be an embedded handlebody. 
%Let $V$. 
%We call the group $\MCG(S^3, V)$ ($\MCG_+(S^3, V)$, respectively) 
%the {\it symmetry group} (the {\it positive symmetry group}, 
%respectively) of $V$. 

\subsection{Relation between the groups $\MCG_+(S^3, V)$ and $\MCG_+(E(V))$.}
\label{sec:Relation between the groups MCG+(S3, V) and MCG+(E(V)).}

Let $(S^3, V)$ be a handlebody-knot. 
\begin{lemma}
\label{lem:injectivity of the map from MCG_+(S3, V) to MCG+(E(V))}
The natural homomorphism from $\MCG_+(S^3, V)$ to $\MCG_+(E(V))$ 
which takes $f \in \MCG_+(S^3, V)$ to
$f|_{E(V)} \in \MCG_+ (E(V))$ is an injection. 
\end{lemma}
\begin{proof}
The result follows immediately from the well-known fact 
that $\MCG(V ~\rel~ \partial V)$ is the trivial group, 
see e.g. Fomenko-Matveev \cite[Theorem 3.7]{FM97}. 
\end{proof}

The mapping class group $\MCG_+(E(V))$ is finitely presented 
because of the following theorem. 

\begin{theorem}[McCullough \cite{McC91}]
\label{thm:finitely presentability by McCullough}
Mapping class groups of irreducible 
orientable sufficiently large 
$3$-manifolds $($possibly not 
boundary-irreducible$)$ are finitely presented. 
\end{theorem}

Following this theorem, Lemma 
\ref{lem:injectivity of the map from MCG_+(S3, V) to MCG+(E(V))} 
implies that the symmetry group $\MCG_+(S^3, V)$ can be 
regarded as a subgroup of the finitely presented group $\MCG_+(E(V))$. 
However, a subgroup of a finitely generated (presented, respectively) 
group is not always finitely generated (presented, respectively). 

\subsection{Symmetry groups of knots}
\label{sec:Symmetry groups of knots}

Let $K$ be a knot in the 3-sphere $S^3$. 
Then the group $\MCG (S^3 , K) $ has been called the  
{\it symmetry group} 
of the knot $K$ and often denoted by 
$\Sym (S^3, K)$. 
By the uniqueness of regular neighborhoods, 
we have $\Sym (S^3, K) \cong \MCG (S^3 , N(K))$. 

The following proposition is a direct 
consequence of Gordon-Luecke \cite{GL89} 
that states that if two knots in $S^3$
have homeomorphic complements then the homeomorphism between the
two complements extends to an automorphism of $S^3$.

\begin{proposition}
\label{lem:isomorphism in the case of knots}
For a non-trivial knot $K$ in $S^3$, 
there exist the following natural isomorphisms$:$ 
\[ \MCG (S^3 , N(K)) \cong \MCG(E(K)) ,~ 
\MCG_+ (S^3 , N(K)) \cong \MCG_+(E(K)). \] 
\end{proposition}

By Proposition 
\ref{lem:isomorphism in the case of knots} and 
Theorem \ref{thm:finitely presentability by McCullough}, 
we have the following 
(the case of the unknot is clear)$:$
\begin{corollary}
For a knot $K$ in $S^3$, the symmetry group 
$\MCG_+ (S^3 , N(K))$ 
is finitely presented. 
\end{corollary}

\subsection{Goeritz group of the trivial handlebody-knot of genus two}
\label{sec:Goeritz group of the trivial handlebody-knot of genus two}

Let $(S^3, V)$ be the trivial handlebody-knot of 
genus $g$. 
Then the group $\MCG_+(S^3, V)$ is called 
the {\it Goeritz group} of genus $g$. 
As mentioned in the Introduction, the case of genus two 
was studied by Goeritz \cite{Goe33}, Scharlemann \cite{Sch04}, 
Akbas \cite{Akb08} and Cho \cite{Cho08}. 
We put these results together in the following theorem. 
\begin{theorem}[\cite{Akb08, Cho08, Goe33, Sch04}]
\label{thm:genus two Goeritz group is finitely presented}
In the case where $(S^3, V)$ is the trivial handlebody-knot of genus two, 
the mapping class group $\MCG_+(S^3, V)$ is finitely presented. 
\end{theorem}
%See \cite{Cho08} for a detailed history and references. 

\section{Symmetry groups of spatial graphs}
\label{sec:Symmetry groups of spatial graphs}

Recall that by a graph 
we mean the underlying 
space of a finite simplicial complex of dimension 
at most one. 
A point $v$ in a graph is called 
a {\it vertex} if 
$v$ does not have an open neighborhood that 
is homeomorphic to 
an open interval. 
We denote by $v (\Gamma)$ the set of the 
vertices of a graph $\Gamma$. 
The closure of each component of 
$\Gamma \setminus v(\Gamma)$ is called 
an {\it edge}. 
We denote by $e (\Gamma)$ the set of the 
edges of a graph $\Gamma$. 
Throughout the paper, 
all graphs are assumed to be connected and 
all graphs are assumed not to have
any valency-1 vertices, i.e. any vertex admits no
open neighborhood homeomorphic to $[0,1)$.
As for the handlebody-knots, 
we will denote by $(S^3, \Gamma)$ 
a spatial graph $\Gamma$ in $S^3$.

Let $(S^3, \Gamma)$ be a spatial graph. 
We call the group $\MCG (S^3 , \Gamma) $ the 
{\it symmetry group} of $(S^3, \Gamma)$. 
Note that the subgroup of the simplicial automorphism group 
of $\Gamma$ that is induced by 
homeomorphisms of the pair $(S^3, \Gamma)$ 
is called the 
{\it topological symmetry group} of the spatial graph 
$(S^3, \Gamma)$ and often denoted by $\TSG (S^3, \Gamma)$.  
This group was introduced by Simon \cite{Sim86}. 
See Flapan-Naimi-Pommersheim-Tamvakis \cite{FNPT05} for details. 
Obviously, the group $\TSG (S^3, \Gamma)$ is 
a finite group. 
We note that, if $\Gamma$ is homeomorphic to a circle 
(i.e. if $(S^3, \Gamma)$ is a knot in $S^3$), we have the following$:$
\begin{eqnarray*}
\TSG (S^3, \Gamma) = \left\{ 
\begin{array}{ll}
1 & \mbox{if $(S^3, \Gamma)$ is not invertible},\\
\Integer / 2 \Integer & \mbox{if $(S^3, \Gamma)$ is invertible}.\\
\end{array}
\right.
\end{eqnarray*}

The following proposition is a straightforward consequence of the definitions. 

\begin{lemma}
\label{pro:symmetry groups and topological symmetry groups}
The group $\MCG(S^3 ~\rel~ \Gamma)$ is a normal subgroup of 
the symmetry group $\MCG(S^3, \Gamma)$ and 
$\TSG (S^3, \Gamma) \cong 
\MCG(S^3, \Gamma) / \MCG(S^3 ~\rel~ \Gamma)$. 
Hence $\MCG(S^3, \Gamma) \cong \TSG(S^3, \Gamma)$ if and only if 
$\MCG(S^3 ~\rel~ \Gamma) \cong 1$.  
\end{lemma}

\subsection{Relation between symmetry groups of spatial graphs and handlebody-knots}
\label{subsec:Symmetry groups of spatial graphs and handlebody-knots}

Let $V$ be a handlebody and let $D_1 , D_2 , \ldots , D_n$ be 
mutually disjoint, mutually non-parallel 
essential disks in $V$. 
Suppose that $V$ cut off by $\bigcup_{i=1}^n D_i$ 
consists of 3-balls. 
Then there exists a graph $\Gamma$ embedded in $\Int \thinspace V$ 
such that 
\begin{itemize}
\item
$\Gamma$ is a deformation retract of $V$, 
\item
$\Gamma$ intersects $\bigcup_{i=1}^n D_i$ 
transversely at points in the interior of the 
edges of $\Gamma$, and   
\item
$\Gamma$ has exactly $n$ edges 
$e_1 , e_2 , \ldots , e_n$ and 
$ \# ( e_i \cap D_j ) = \delta_{ij}$, where 
$\delta_{ij}$ is the Kronecker delta. 
\end{itemize}
We call the above graph $\Gamma$ 
the {\it dual graph} of $\{D_1 , D_2 , \ldots , D_n \}$. 
Note that such a graph is uniquely determined up to isotopy. 

%\begin{lemma}
%\label{lem:triviality of a mapping class group}
%Let $V$ be a handlebody and let $\Gamma$ be a graph embedded 
%in $V$ such that $\Gamma$ is a deformation retract of $V$. 
%Then the group $\MCG (V , \Gamma ~\rel~ \partial V)$ 
%is trivial.  
%\end{lemma}
%\begin{proof}
%Let $\{ D_1 , D_2 , \ldots , D_n \}$ be the family of 
%mutually disjoint, mutually non-parallel 
%essential disks in $V$ such that 
%$\Gamma$ is the dual graph of 
%$\{ D_1 , D_2 , \ldots , D_n \} $. 
%Let $f \in \Aut (V , \Gamma ~\rel~ \partial V)$. 
%Since $V$ is irreducible, we can isotope $f$ 
%so that $f(D_i) = D_i$ for 
%$1 \leqslant i \leqslant n$. 
%Moreover, $f$ can be isotoped so that 
%$f|_{D_i} = \id_{D_i}$ by Alexander's trick. 
%Now, the lemma follows from Alexander's trick 
%for the complementary balls of $\bigcup_{i=1}^n D_i$. 
%\end{proof}

Let $(S^3, \Gamma)$ be a spatial graph. 
Let $\{ D_1 , D_2 , \ldots , D_n \}$ be the family of 
mutually disjoint, mutually non-parallel 
essential disks in $N(\Gamma)$ such that 
$\Gamma$ is the dual graph of 
$\{ D_1 , D_2 , \ldots , D_n \} $. 
It follows from Alexander's trick that 
$\MCG(S^3, \Gamma)$ is isomorphic to 
$\MCG(S^3, N(\Gamma), \bigcup_{i=1}^n D_i)$. 
On the other hand, there is a natural  
injective homomorphism from 
$\MCG(S^3, N(\Gamma), \bigcup_{i=1}^n D_i)$ 
to 
$\MCG(S^3, N(\Gamma))$. 
Consequently, we have the following proposition.

\begin{proposition}
\label{thm:the symmetry group of a spatial graph is a subgroup}
There exists a natural injective  homomorphism 
from $\MCG (S^3 , \Gamma)$ to $\MCG(S^3 , N(\Gamma))$. 
\end{proposition}

Proposition \ref{thm:the symmetry group of a spatial graph is a subgroup} 
implies that the symmetry group of a spatial graph 
can be regarded as a subgroup of the symmetry group of its regular 
neighborhood. 

\subsection{Mapping class groups of $3$-manifolds with boundary-pattern}
\label{subsec:Mapping class groups of 3-manifolds with boundary-pattern}

We review the notion of boundary-pattern defined 
in Johannson \cite{Joh79}. 
Let $M$ be a compact $3$-manifold. 
A {\it boundary-pattern} for $M$ consists of a set 
$\underline{\underline{m}}$ of finitely many 
compact connected surfaces in 
$\partial M$, such that the intersection of any $i$ of them is 
empty or consists of $(3 - i)$-manifolds. 
A boundary-pattern is said to be {\it complete} when 
$\bigcup_{A \in \underline{\underline{m}}} A = \partial M$. 

We denote by $(M, \underline{\underline{m}})$ a 3-manifold 
with a boundary-pattern $\underline{\underline{m}}$. 

A boundary-pattern $\underline{\underline{m}}$ 
of a 3-manifold $M$ is said to be {\it useful} 
if the boundary of any disk $D$ properly embedded in $M$, where 
$\partial D$ intersects $\partial A$ 
transversely for each $A \in \underline{\underline{m}}$ and 
$\# (D \cap ( \bigcup_{A \in \underline{\underline{m}}} 
\partial A)) \leqslant 3$, 
bounds a disk $E$ in $\partial M$ such that 
$E \cap ( \bigcup_{A \in \underline{\underline{m}}} \partial A )$ 
is the cone on 
$\partial E \cap ( \bigcup_{A 
\in \underline{\underline{m}}} \partial A ) $. 
An {\it essential} disk in $(M, \underline{\underline{m}})$ 
is an essential disk $D$ in $M$ such that 
$D \cap ( \bigcup_{A \in \underline{\underline{m}}} 
\partial A ) = \emptyset$.

Let $(M, \underline{\underline{m}})$ a 3-manifold 
with a boundary-pattern, 
where $\underline{\underline{m}} = \{A_1, A_2, \ldots, A_k \}$.
We denote the group 
$\MCG (M , A_1 , A_2 , \ldots , A_k)$ 
by 
$\MCG(M , \underline{\underline{m}})$.  
{\it Dots} of $(M , \underline{\underline{m}})$ are 
pairwise disjoint points $p_1 , p_2 , \ldots , p_l$ in 
$\partial M$ such that 
$p_i \cap ( \bigcup_{A \in \underline{\underline{m}}} 
\partial A ) = \emptyset$. 
We denote the group 
$\MCG (M , A_1 , A_2 , \ldots , A_k , p_1 , p_2 , \ldots , p_l)$ 
by 
$\MCG (M , \underline{\underline{m}} , p_1 , p_2 , \ldots , p_l)$. 
{\it Spots} of $(M , \underline{\underline{m}})$ are 
pairwise disjoint disks $D_1 , D_2 , \ldots , D_k$ in 
$\partial M$ such that 
$D_i \cap ( \bigcup_{A \in \underline{\underline{m}}} 
\partial A ) = \emptyset$. 
We denote the group 
$\MCG (M , A_1 , A_2 , \ldots , A_k ~\rel~  D_1 \cup D_2 \cup \cdots \cup D_l )$  
by 
$\MCG (M , \underline{\underline{m}} ~\rel~  D_1 \cup D_2 \cup \cdots \cup D_l )$.

\begin{lemma}
\label{lem:mapping class group of a 3-manifold with dotted boundary-pattern} 
Let $(M , \underline{\underline{m}})$ be a compact 
$3$-manifold with boundary-pattern such that 
$\partial M \ncong S^1 \times S^1$ and 
$\bigcup_{A \in \underline{\underline{m}}} \partial A$ consists of 
non-contractible circles on $\partial M$. 
Let $p_1$ and $p_2$ be dots of 
$(M , \underline{\underline{m}})$ contained in  
planar components $A_1$ and $A_2$ $($possibly $A_1 = A_2)$ 
of $\underline{\underline{m}}$, respectively. 
%Set $D_i = N(p_i ; A)$ for $1 \leqslant i \leqslant k$. 
Set $D_i = N(p_i ; A_i)$ for $i= 1, 2$. 
Suppose that $\MCG (M , \underline{\underline{m}})$ is 
finitely presented. 
%Then both $\MCG(M , \underline{\underline{m}} , 
%p_1 , p_2 , \ldots , p_k )$ and 
%$\MCG(M , \underline{\underline{m}} , 
%D_1 , D_2 , \ldots , D_k )$ 
Then all of $\MCG(M , \underline{\underline{m}} , 
p_1)$, $\MCG(M , \underline{\underline{m}} , 
p_1 , p_2)$, 
$\MCG(M , \underline{\underline{m}} ~\rel~ D_1)$ 
and 
$\MCG(M , \underline{\underline{m}} ~ 
\rel ~  D_1 \cup D_2 )$ 
are finitely presented. 
\end{lemma}
\begin{proof}
%Fix a set of generators  
%$\{ \gamma_1 , \gamma_2 , \ldots , \gamma_g \}$ 
%of the group $\pi_1(A_1 , p_1)$. 
Let $P_1 : \pi_1(A_1 , p_1) \to \MCG(M , \underline{\underline{m}} , p_1)$ 
be the {\it point-pushing map}. 
That is, let 
$\gamma : [0, 1] \to S$ be a loop in $A_1$ based 
at $p_1$. 
There is an isotopy $\varphi_t : M \to M$ 
supported in a small neighborhood of
the loop $\gamma$ such that $\varphi_0 = \id$, and 
$\varphi_t(p_1) = \gamma(t)$. 
The map $P_1$ is then defined by 
$P_1(\gamma) = \varphi_1$. 
%Let $G = \langle 
%\gamma_1 , \gamma_2 , \ldots , \gamma_g 
%\mid \gamma_1 \gamma_2 \cdots \gamma_g \rangle , \gamma_{2i-1}{\gamma_{2i}}^{-1} 
%i=1 , 2 , \ldots , g_0 \rangle$. 
%Then the homomorphism $P_1$ induces a homomorphism 
%$\bar{P}_1 : G 
%\to \MCG (M , \underline{\underline{m}} , p_1)$. 
Since $\partial M$ is connected and 
$\partial M$ is neither $S^2$ nor $S^1 \times S^1$, 
$\partial M$ does not admit a non-trivial $S^1$ action. 
Therefore as a variation of the Birman exact sequence \cite{Bir74, HU12}, we have 
the following exact sequence$:$ 
\begin{eqnarray*}
1 \to \pi_1(A_1 , p_1) 
\xrightarrow{ P_1 } \MCG (M , \underline{\underline{m}} , p_1) 
\to \MCG (M , \underline{\underline{m}} ) \to 1 .
\end{eqnarray*}
This sequence proves that $\MCG (M , \underline{\underline{m}} , p_1)$ 
is finitely presented by Lemma \ref{lem:exact sequence and finite presentation}. 

%Let $k > 1$. 
%We denote by $\mathrm{Conf}_{k-1}(A \setminus \{ p_1 \})$, 
%the configuration space of $k-1$ points in $A \setminus \{ p_1 \}$, 
%i.e. 
%\[\Conf_{k-1} ( A \setminus \{ p_1 \} ) = \{ (q_2 , q_3 , \ldots , q_k) 
%\in (A \setminus \{ p_1 \}) ^{k-1} 
%\mid  q_i \neq q_j \mbox{ if } i \neq j \} . \]  
%
%\[ 1 \to \pi_1(\Conf_{k-1}(A \setminus \{ p_1 \})) \to 
%\MCG (M, \underline{\underline{m}} , p_1, p_2, \ldots , p_k) \to 
%\langle \tau \rangle  \to 1 , 
%\]
%Where $\tau$ is the Dehn
%twist about a disk on $A \setminus \{ p_1 \}$. 
%$\pi_1(\Conf_{k-1}(A))$ is finitely presented since 
%$\Conf_{k-1}(A)$ is compact. 

%Fix a set of generators 
%$\{ \gamma_1^\prime , \gamma_2^\prime , \ldots , 
%\gamma_{g+1}^\prime \}$ 
%of the group $\pi_1(A \setminus \{ p_1 \} , p_2)$. 
%where $\{ \gamma_1^\prime , \gamma_2^\prime , \ldots , \gamma_{g}^\prime  \}$
%corresponds to $\{ C_1 , C_2 , \ldots , C_g \}$ 
%and $\gamma_{g+1}^\prime$ corresponds to 
%$\partial D_1$. 
%Note that $\pi_1(A , p_1) = \langle 
%\gamma_1^\prime , \gamma_2^\prime , \ldots , \gamma_{g+1}^\prime
%\mid \gamma_1^\prime \gamma_2^\prime \cdots  \gamma_{g+1}^\prime 
%\rangle$. 
If $A_1 \neq A_2 $, then 
it is clear from the above argument that 
$\MCG (M, \underline{\underline{m}} , p_1, p_2 )$ is finitely presented. 
Let $A_1 = A_2 = A$ and $P_2 : \pi_1(A \setminus \{ p_1 \} , p_2) 
\to \MCG(M , \underline{\underline{m}} , p_1 , p_2)$ 
be the point-pushing map. 
%Let $G^\prime = \langle 
%\gamma_1^\prime , \gamma_2^\prime , \ldots , 
%\gamma_{g+1}^\prime 
%\mid 
%\gamma_1^\prime \gamma_2^\prime \cdots  \gamma_{g+1}^\prime , 
%\gamma_{2i-1}^\prime{\gamma_{2i}^\prime}^{-1} 
%i=1 , 2 , \ldots , g_0 
%\rangle$. 
%The map $P_2$ induces a map 
%$\bar{P}_2 : \pi_1(A \setminus \{ p_1 \} , p_2)  
%\to \MCG (M , \underline{\underline{m}} , p_1 , p_2)$. 
Then the Birman exact sequence 
\[ 1 \to 
\pi_1(A \setminus \{ p_1 \} , p_2) \xrightarrow{P_2} 
\MCG (M, \underline{\underline{m}} , p_1, p_2 ) \to 
\MCG (M, \underline{\underline{m}} , p_1)  \to 1 , 
\]
shows that  
$\MCG (M, \underline{\underline{m}} , p_1, p_2 )$ is finitely presented.

Let $\tau_1 \in \MCG(M , \underline{\underline{m}} ~\rel~ D_1 )$ be the Dehn twist about $D_1$. 
Then we have the following exact sequence$:$ 
\[
1 \to \langle \tau_1 \rangle \to \MCG (M , \underline{\underline{m}}  ~\rel~ D_1) 
\to \MCG (M , \underline{\underline{m}} , p_1) \to 1, 
\]
where the map $ \MCG (M , \underline{\underline{m}} ~\rel~ D_1) 
\to \MCG (M , \underline{\underline{m}} , p_1)$ 
is induced by collapsing the spot $D_1$ to $p_1$. 
This implies that 
$\MCG (M , \underline{\underline{m}}  ~\rel~D_1)$ is finitely presented. 

Finally, let $\tau_i^\prime \in \MCG(M , \underline{\underline{m}} ~\rel~ 
D_1 \cup D_2)$ be the Dehn twist about $D_i$ for $i=1,2$. 
Considering the exact sequence  
\[
1 \to \langle \tau_1 , \tau_2 \rangle \to 
\MCG (M , \underline{\underline{m}}  ~\rel~ D_1 \cup D_2) 
\to \MCG (M , \underline{\underline{m}} , p_1 , p_2) \to 1, 
\]
we see that $\MCG (M , \underline{\underline{m}} ~\rel~ D_1 \cup D_2)$ 
is finitely presented. 
\end{proof}

\begin{lemma}
\label{lem:mapping class group of a 3-manifold 
with dotted boundary-pattern 
whose dots are in a disk or an annulus}
Let $(M , \underline{\underline{m}})$ be a compact 
$3$-manifold with boundary-pattern such that 
$\partial M \cong S^1 \times S^1$ 
and 
$\bigcup_{A \in \underline{\underline{m}}} \partial A$ consists of 
non-contractible circles on $\partial M$. 
%Let $p_1 , p_2 , \ldots , p_k$ be dots of 
Let $p_1$ and $p_2$ be dots of 
$(M , \underline{\underline{m}})$ contained in  
components $A_1$ and $A_2$ $($possibly $A_1 = A_2)$ 
of $\underline{\underline{m}}$, respectively. 
%Set $D_i = N(p_i ; A)$ for $1 \leqslant i \leqslant k$. 
Set $D_i = N(p_i ; A_i)$ for $i= 1, 2$. 
Suppose that $\MCG (M , \underline{\underline{m}})$ is 
finitely presented. 
%Then both $\MCG(M , \underline{\underline{m}} , 
%p_1 , p_2 , \ldots , p_k )$ and 
%$\MCG(M , \underline{\underline{m}} , 
%D_1 , D_2 , \ldots , D_k )$ 
Then all of $\MCG(M , \underline{\underline{m}} , 
p_1)$, $\MCG(M , \underline{\underline{m}} , 
p_1 , p_2)$, 
$\MCG(M , \underline{\underline{m}} ~\rel~ D_1)$ 
and 
$\MCG(M , \underline{\underline{m}} ~\rel~  D_1 \cup D_2  )$ 
are finitely presented. 
\end{lemma}
\begin{proof}
It is clear that $\underline{\underline{m}}$ consists 
of annuli. 
We give the proof for the case $A_1 = A_2 = A$; 
the other case is left to the reader. 
Fix a generator 
$\gamma$ of the group $\pi_1(A, p_1)$. 
Let $P_1^\prime : \pi_1(A , p_1) 
\to \MCG(M , \underline{\underline{m}} , p_1)$ 
be the point-pushing map. 
The map $P_1^\prime$ induces a map 
$\bar{P}_1^\prime : \pi_1(A , p_1) 
/ \ker P_1^\prime 
\to \MCG (M , \underline{\underline{m}} , p_1)$. 
Then there exists the Birman exact sequence 
\[ 1 \to \pi_1 (A , p_1 ) / 
 \ker P_1^\prime \xrightarrow{\bar{P}_1^\prime} 
\MCG (M, \underline{\underline{m}} , p_1 ) \to 
\MCG (M, \underline{\underline{m}})  \to 1 . 
\]
Since $\pi_1 (A , p_1 ) = \langle \gamma \rangle \cong \Integer$, 
the group $\pi_1 (A, p_1 ) / 
 \ker P_1^\prime$ is finitely presented. 
Hence $\MCG (M , \underline{\underline{m}} , p_1)$ 
is finitely presented by Lemma \ref{lem:exact sequence and finite presentation}. 

%Fix a set of generators 
%$\{ \gamma_1 , \gamma_2 \}$ of 
%of the group $\pi_1(A \setminus \{ p_1 \} , p_2)$ 
%that corresponds to the circles $\partial A$. 
Let $P_2^\prime : \pi_1(A \setminus \{ p_1 \} , p_2) 
\to \MCG(M , \underline{\underline{m}} , p_1 , p_2)$ 
be the point-pushing map. 
%The map $P_2$ induces a map 
%$\bar{P}_2 : \langle \gamma_1 \rangle 
%\to  \MCG (M , \underline{\underline{m}} , p_1 , p_2)$. 
Then by the Birman exact sequence 
\[ 1 \to \pi_1 (A \setminus \{ p_1 \} , p_2 ) 
\xrightarrow{P_2^\prime} 
\MCG (M, \underline{\underline{m}} , p_1, p_2 ) \to 
\MCG (M, \underline{\underline{m}} , p_1 )  \to 1  
\]
it follows that $\MCG (M , \underline{\underline{m}} , p_1, p_2)$ 
is finitely presented by 
Lemma \ref{lem:exact sequence and finite presentation}. 

The proof for the mapping class group of spotted 3-manifolds with 
boundary-pattern is the same as in Lemma 
\ref{lem:mapping class group of a 3-manifold with dotted boundary-pattern}.  
\end{proof}

\subsection{Type I spheres for spatial graphs}
\label{subsec:Type I spheres for spatial graphs}

Let $(S^3, \Gamma)$ be a spatial graph. 
Let $P$ be a 2-sphere embedded in $S^3$ satisfying$:$
\begin{itemize}
\item
the sphere $P$ intersects $\Gamma$ in a single vertex, and
\item
each of the two
components of $S^3 \setminus P$ contains
non-empty part of $\Gamma$.
\end{itemize}
Then $P$ is called a {\it type I sphere} for $(S^3, \Gamma)$. 
Two Type I spheres $P_1$ and $P_2$ 
are said to be {isotopic} if 
there exists an ambient isotopy fixing $\Gamma$ 
that moves $P_1$ to $P_2$. 
Two Type I spheres $P_1$ and $P_2$ for $(S^3, \Gamma)$ are said to be 
{\it weakly disjoint} if 
$(P_1 \cap P_2) \setminus  \Gamma = \emptyset$.

\begin{lemma}
\label{lem:orbits of vertices}
Let $(S^3, \Gamma)$ be a spatial graph. 
Let $P_1$ and $P_2$ be Type I spheres for $(S^3, \Gamma)$. 
Then there exists an element $f \in \Aut_+(S^3 ~\rel~ \Gamma)$ such that 
$f(P_1) = P_2$ if and only if 
$P_1$ and $P_2$ bound $3$-balls $B_1$ and $B_2$, respectively, so that  
$\thinspace B_1 \cap \Gamma = B_2 \cap \Gamma$. 
\end{lemma}
\begin{proof}
The ``only if" part is trivial. 
Suppose that there exist $3$-balls $B_1$ and $B_2$ 
bounded by $P_1$ and $P_2$, respectively, such that  
$\thinspace B_1 \cap \Gamma = B_2 \cap \Gamma$. 
Set $\Gamma' = B_1 \cap \Gamma$ and $\Gamma'' = E(B_1) \cap \Gamma$. 
Since $P_1$ and $P_2$ are bounding 3-balls $E(B_1)$ and $E(B_2)$, respectively, 
that do not intersect $\Gamma_1$ in their interiors, 
there exists an ambient isotopy in $S^3$ that 
fixes $\Gamma_1$ and that moves $P_1$ to $P_2$. 
This isotopy determines an orientation-preserving homeomorphism 
$f_1: (B_1, \Gamma') \to (B_2, \Gamma')$. 
In the same way, we see that 
there exists an orientation-preserving homeomorphism 
$f_2: (E(B_1), \Gamma'') \to (E(B_2), \Gamma'')$. 
Since the mapping class group of the sphere with one dot 
is the trivial group, we get the required element 
$f \in \Aut_+(S^3 ~\rel~ \Gamma)$ 
from $f_1$ and $f_2$. 
\end{proof}

\begin{lemma}
\label{lem:orbits of edges}
Let $(S^3, \Gamma)$ be a spatial graph. 
For $i=1,2$, let $P_i$ and $P'_i$ be non-isotopic, weakly disjoint 
Type I spheres for $(S^3, \Gamma)$. 
Then there exists an element $f \in \Aut_+(S^3 ~\rel~ \Gamma)$ such that 
$f(P_1) = P_2$ and $f(P'_1) = P'_2$ if and only if 
$P_1$, $P'_1$, $P_2$ and $P'_2$ bound $3$-balls 
$B_1$, $B'_1$, $B_2$ and $B'_2$, respectively, so that 
$B_1 \cap B'_1 = \emptyset$, 
$B_1 \cap \Gamma = B_2 \cap \Gamma$ and $B'_1 \cap \Gamma = B'_2 \cap \Gamma$. 
\end{lemma}
\begin{proof}
The ``only if" part is trivial. 
Suppose that 
$P_1$, $P'_1$, $P_2$ and $P'_2$ bound $3$-balls 
$B_1$, $B'_1$, $B_2$ and $B'_2$, respectively, so that  
$B_1 \cap \Gamma = B_2 \cap \Gamma$ and $B'_1 \cap \Gamma = B'_2 \cap \Gamma$. 
By Lemma \ref{lem:orbits of vertices}, 
there exists an element $g \in \Aut_+(S^3 ~\rel~ \Gamma)$ such that 
$g(B_1) = B_2$. 
Since $B_1$ and $B_2$ are contractible in $S^3$, 
we can again apply an almost the same argument for 
$g(B'_1)$ and $B'_2$ as in the proof of 
Lemma \ref{lem:orbits of vertices} 
to get the required 
element 
$f \in \Aut_+(S^3 ~\rel~ \Gamma)$ such that 
$f(P_1) = P_2$ and $f(P'_1) = P'_2$. 
\end{proof}

Let $P$ and $P'$ be non-isotopic, weakly disjoint 
Type I spheres for a spatial graph $(S^3, \Gamma)$. 
Then $P$ and $P'$ bound 3-balls $B$ and $B'$, respectively, 
so that $\Int \thinspace B_1 \cap \Int \thinspace B_2 = \emptyset$. 
These 3-balls determines 
a partition of the set $e(\Gamma)$ edges of 
$\Gamma$ into three blocks as  
$\{ e \in e(\Gamma) \mid e \subset B_i \}$, 
$\{ e \in e(\Gamma) \mid e \subset B'_i \}$ and 
$\{ e \in e(\Gamma) \mid e \subset E(B_i \cup B'_i) \}$. 
None of the three blocks is empty 
by the definition of Type I spheres and 
the assumption that $P$ and $P'$ are non-isotopic. 
We remark that the partition depends only on 
the pair $(P, P')$. 

%We remark that, in the statement of 
%Lemma \ref{lem:orbits of edges}, 
%we can suppose further that 
%$\Int \thinspace B_1 \cap \Int \thinspace B'_1 = 
%\Int \thinspace B_2 \cap \Int \thinspace B'_2 = \emptyset$ 
%by replacing some of $B_1$, $B'_1$, $B_2$ and $B'_2$ by 
%$E(B_1)$, $E(B'_1)$, $E(B_2)$ and $E(B'_2)$, if necessary. 
%Then each pair $(P_i, P'_i)$ determines 
%a partition of the set $e(\Gamma)$ edges of 
%$\Gamma$ into three blocks, 
%$\{ e \in e(\Gamma) \mid e \subset B_i \}$, 
%$\{ e \in e(\Gamma) \mid e \subset B'_i \}$ and 
%$\{ e \in e(\Gamma) \mid e \subset E(B_i \cup B'_i) \}$. 
%None of them is the empty set 
%by the definition of Type I spheres and 
%the assumption that $P_i$ and $P'_i$ are non-isotopic. 
%This partition depends only on 
%the pair $(P_i, P'_i)$. 

\begin{lemma}
\label{lemma:finiteness of orbits}
Let $(S^3, \Gamma)$ be a spatial graph. 
Up to the action of $\Aut_+(S^3 ~ \rel ~ \Gamma)$, 
there exist only finitely many pairs of non-isotopic, 
weakly disjoint Type I spheres for $(S^3, \Gamma)$. 
\end{lemma}
\begin{proof}
For $i=1,2$, let $P_i$ and $P'_i$ be non-isotopic, weakly disjoint 
Type I spheres for $(S^3, \Gamma)$. 
As mentioned above, 
each pair $(P_i , P'_i)$ determines a partition 
of the set of edges of $\Gamma$ into three blocks. 
By Lemma \ref{lem:orbits of edges}, 
the pairs $(P_1, P'_1)$ and $(P_2, P'_2)$ 
coincide up to the action of $\Aut_+(S^3 ~ \rel ~ \Gamma)$ 
if and only if 
the pairs $(P_1 , P'_1)$ and  
$(P_2 , P'_2)$ determine the same partition of the 
edges of $\Gamma$. 
Since there exist only finitely many partitions of the edges of 
$\Gamma$, we get the assertion.  
\end{proof}

\subsection{Finite presentation of the symmetry groups of spatial graphs}
\label{subsec:Finite presentation of the symmetry groups of spatial graphs}

Let $(S^3, K)$ be a knot. 
Let $D_1, D_2, D_3, D_4$ be four mutually disjoint meridian disks of $N(K)$. 
Let $\underline{\underline{m}}$ be the set of 
the closure of each component of 
$\partial N(K)$ cut off by $\bigcup_{i=1}^4 \partial D_i$. 
%Set $A = \partial N(K) \cap N(D ; N(K))$ 
%Regard $A$ 
%as $S^1 \times [0,1]$, and 
%divide it into four annuli 
%$A_1 = S^1 \times [0,1/4]$, $A_2 = S^1 \times [1/4,1/2]$, 
%$A_3 = S^1 \times [1/2,3/4]$ and 
%$A_4 = S^1 \times [3/4,1]$. 
%Let $B = \partial N(K) \setminus \Int A$. 
Then the complete boundary-pattern 
$\underline{\underline{m}}$ satisfies the following$:$ 
\begin{lemma}
\label{lem:the case of the knot complement with boundary-pattern}
$\MCG(E(K) , \underline{\underline{m}})$ is 
finitely presented. 
\end{lemma} 
\begin{proof}
It is clear that $\MCG (E(K), \underline{\underline{m}})$ is 
isomorphic to $\MCG (E(K) ~\rel~\partial E(K))$. 
Hence the assertion follows from McCullough \cite{McC91}. 
\end{proof} 

Let $(S^3, \Gamma)$ be a spatial graph. 
Set $V = N(\Gamma)$. 
Let 
\[
\Delta = \{ D_1 , D_2 , \ldots , D_{n_1}, 
D_1^\prime , D_2^\prime , \ldots , D_{n_2}^\prime, 
D_1^{\prime \prime} , D_2^{\prime \prime} , \ldots , D_{n_3}^{\prime\prime} 
 \} 
\] 
be the family of 
essential disks in $V$ such that 
\begin{itemize}
\item 
$\Gamma$ is the dual graph of $\Delta$, 
\item
$\{D_1, D_2, \ldots , D_{n_1}\}$ corresponds to the set of 
loops of $\Gamma$, 
\item
$\{D_1^\prime, D_2^\prime, \ldots , D_{n_2}^\prime\}$ 
corresponds to the set of cut-edges of $\Gamma$. 
\item
$\{D''_1, D''_2, \ldots , D''_{n_3} \}$  
corresponds to the set of edges of $\Gamma$ 
each of which is neither a loop nor a cut-edge of $\Gamma$. 
\end{itemize}
For each $1 \leqslant i \leqslant n_1$, 
we take mutually disjoint parallel copies 
$D_{i,1}$, $D_{i,2}$, $D_{i,3}$, $D_{i,4}$ of $D_i$, and 
for each $1 \leqslant i \leqslant n_3$, 
we take disjoint parallel copies 
$D_{i,1}^{\prime \prime}$, $D_{i,2}^{\prime \prime}$ 
of $D_i^{\prime \prime}$.  
Let 
\[ \Delta^\prime = 
\{ D_{i, j} \mid 1 \leqslant i \leqslant n_1, j= 1,2,3,4 \}
\cup 
\{ D_1^\prime, D_2^\prime, \ldots, D_{n_2}^\prime \} 
\cup \{ D_{i, j}^{\prime \prime} \mid 
1 \leqslant i \leqslant n_3, j= 1,2 \} .
\] 
Let $\underline{\underline{m}}$ be the set of 
components of 
$\partial N(\Gamma)$ cut off by $\bigcup_{D \in \Delta^\prime} \partial D$. 
See Figure \ref{fig:graph_and_boundary-pattern}.  
\begin{figure}[htbp]
\begin{center}
\includegraphics[width=12cm,clip]{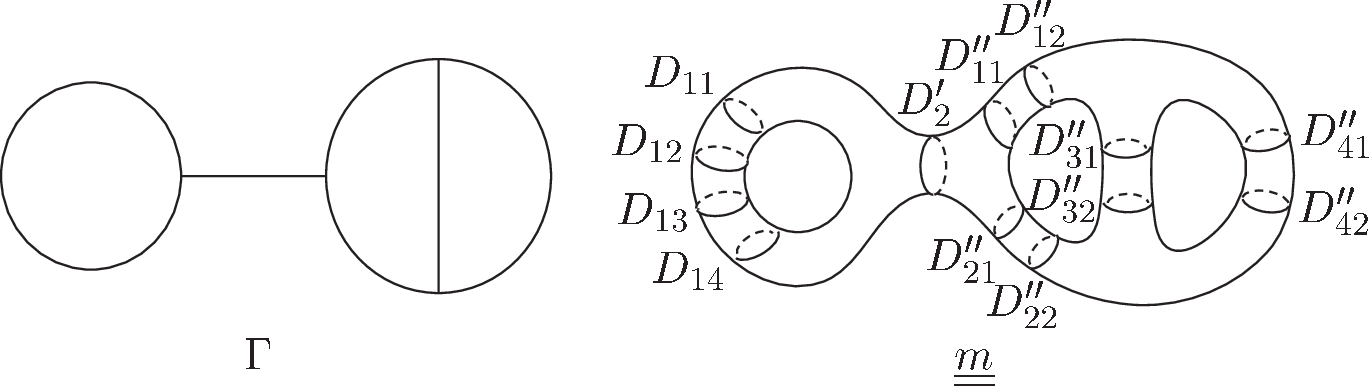}
\caption{A spatial graph $\Gamma$ and the corresponding 
boundary-pattern $\underline{\underline{m}}$.}
\label{fig:graph_and_boundary-pattern}
\end{center}
\end{figure}
%Set $A = \partial N(K) \cap N(D ; N(K))$ 
%Regard $A$ 
%as $S^1 \times [0,1]$, and 
%divide it into four annuli 
%$A_1 = S^1 \times [0,1/4]$, $A_2 = S^1 \times [1/4,1/2]$, 
%$A_3 = S^1 \times [1/2,3/4]$ and 
%$A_4 = S^1 \times [3/4,1]$. 
%Let $B = \partial N(K) \setminus \Int A$. 
Then $\underline{\underline{m}}$ 
is a complete and useful boundary-pattern of 
$E(\Gamma) = E(V)$ if there is no essential disk in 
$(E(\Gamma), \underline{\underline{m}})$.

%Set $A_i = \partial V \cap N(D_i ; V)$ 
%for $1 \leqslant i \leqslant n$. 
%Let $B_1, B_2 , \ldots , B_m$ 
%be the set of components of 
%$\partial V \setminus \bigcup_{i=1}^n \Int \thinspace A_i$.  

%For each essential disk $D_i$ in $N(\Gamma)$ corresponding to 
%a loop $e_i$ of $\Gamma$,  
%regard $N(D_i ; N(\Gamma)) \cap \partial N(\Gamma)$ 
%as $S^1 \times [0,1]$, and 
%divide it into three annuli 
%$S^1 \times [0,1/3]$, $S^1 \times [1/3,2/3]$ and 
%$S^1 \times [2/3,1]$. 
%Applying this operation for 
%every loop of $\Gamma$, 
%we obtain a new boundary-pattern 
%$\underline{\underline{n}}$ of $E(\Gamma)$. 

\begin{lemma}
\label{lem:symmetry group of an irreducible 
3-manifold with useful boundary-pattern}
$\MCG(E(\Gamma), \underline{\underline{m}})$ 
is finitely presented. 
\end{lemma}

The following theorem by Brown plays a key role in 
the proof of Lemma \ref{lem:symmetry group of an irreducible 
3-manifold with useful boundary-pattern}.  
\begin{theorem}[Brown \cite{Bro84}]
\label{thm:theorem by Brown}
Suppose that a group $G$ acts cellularly on a contractible CW-complex
$\mathcal{K}$. 
If the $2$-skeleton of $\mathcal{K}/G$ is finite, 
the stabilizer of each $0$-cell 
of $\mathcal{K}$ is finitely presented, 
and the stabilizer of each $1$-cell of $\mathcal{K}$ is finitely
generated, then $G$ is finitely presented. 
\end{theorem}

\begin{proof}[Proof of Lemma~$\ref{lem:symmetry group of an irreducible 
3-manifold with useful boundary-pattern}$]
Suppose that $E(\Gamma)$ is boundary-irreducible. 
Then $E(\Gamma)$ is a compact, 
irreducible, sufficiently large 3-manifold and 
the boundary-pattern $\underline{\underline{m}}$ of $E(\Gamma)$ 
defined by the system of disks 
\[ \{ D_{i, j} \mid 1 \leqslant i \leqslant n_1, j= 1,2,3,4 \}
\cup 
\{ D_1^\prime, D_2^\prime, \ldots, D_{n_2}^\prime \} 
\cup \{ D_{i, j}^{\prime \prime} \mid 
1 \leqslant i \leqslant n_3, j= 1,2 \} 
\]
in $N(\Gamma)$ as explained just before the proof is useful.  
Hence it is straightforward from McCullough \cite{McC91} that  
$\MCG(E(\Gamma), \underline{\underline{m}})$ 
is finitely presented. 

We now turn to the case where 
$\partial E(\Gamma)$ is compressible 
in $E(\Gamma)$. 
In the following, we include the case where 
$\Gamma$ consists of a single vertex and 
a single loop. 
If $\underline{\underline{m}}$ is still a useful 
boundary-pattern of $E(\Gamma)$, 
then we are done. 
Suppose that there exists a disk $D$ 
properly embedded in $E(\Gamma)$ such that 
$\partial D$ intersects $\partial A$ 
transversely for each $A \in \underline{\underline{m}}$ 
and $\# ( D \cap \bigcup_{A \in \underline{\underline{m}}} 
\partial A) \leqslant 3$. 
Then it is clear from the definition of $\underline{\underline{m}}$ 
that $D$ can be isotoped so that 
$\# ( D \cap \bigcup_{A \in \underline{\underline{m}}} 
\partial A) = 0$. 

The proof is an induction of the genus of $N(\Gamma)$. 
We proved in Lemma 
\ref{lem:the case of the knot complement with boundary-pattern} 
the case where 
the genus of $N(\Gamma)$ is one. 
Assume that the genus $g$ of $N(\Gamma)$ is greater than one 
and that the 
assertion holds for graphs whose regular neighborhoods have 
genus less than $g$. 

We denote by $\mathcal{Z}$ the 
simplicial complex such that 
\begin{itemize}
\item
the set of vertices of $\mathcal{Z}$ 
consists of the isotopy classes of essential disks 
in $(E(\Gamma), \underline{\underline{m}})$, 
\item 
a collection of $k+1$ vertices spans a
$k$-simplex if and only if they admit a 
set of pairwise-disjoint representatives. 
\end{itemize}
Note that $\mathcal{Z}$ is a subcomplex 
of the {\it disk complex} of $E(\Gamma)$,  
see McCullough \cite{McC91} for the definition of disk complexes. 
We remark that 
each annulus of $\underline{\underline{m}}$ is 
incompressible in $E(\Gamma)$ by the definition of 
$\underline{\underline{m}}$. 
This implies that if $E$ is an essential disk in 
$\underline{\underline{m}}$, $\partial E$ 
is contained in a non-annular element of $\underline{\underline{m}})$. 

Let $E_1$ and $E_2$ be distinct essential 
disks in $(E(\Gamma), \underline{\underline{m}})$. 
Assume that 
$\#  ( E_1 \cap E_2 )$ is minimal up to isotopy. 
Suppose that $E_1 \cap E_2 \neq \emptyset$. 
Let $C \subset E_1$ be a disk cut off from $E_1$ by an outermost
arc $\alpha$ of $E_1 \cap E_2$ 
in $E_1$ such that $C \cap E_2 = \alpha$. 
Then both disks
obtained from surgery on $E_2$ 
along $C$ represent vertices 
of $\mathcal{Z}$. 
It follows from the proof of 
Theorem 4.2 in Cho \cite{Cho08} 
that $\mathcal{Z}$ is contractible. 

In order to prove that 
the $2$-skeleton of 
$\mathcal{Z} / \MCG_+(E(\Gamma), \underline{\underline{m}})$ is finite, 
we fix an auxiliary orientation of each cut-edge of $\Gamma$. 
Let $E$ be an essential disk in $(E(\Gamma), \underline{\underline{m}})$. 
If $\partial E$ is isotopic in $\partial E(\Gamma)$ to 
the boundary of one, say $D_1^\prime$, of the separating disks 
$D_1^\prime, D_2^\prime , \ldots, D_{n_2}^\prime$, 
we denote by $\hat{E}$ 
an essential disk in $N(\Gamma)$ such that 
$\partial E = \partial \hat{E}$ and that 
$\hat{E}$ intersects $\Gamma$ at the endpoint of 
the cut-edge of $\Gamma$ corresponding to $D_1^\prime$. 
Let $A$ is an element of $\underline{\underline{m}}$ 
containing $\partial E$. 
If $\partial E$ is parallel to none of the separating disks 
$D_1^\prime, D_2^\prime , \ldots, D_{n_2}^\prime$, 
we denote by $\hat{E}$ 
an essential disk in $N(\Gamma)$ such that 
$\partial E = \partial \hat{E}$ and that 
$\hat{E}$ intersects $\Gamma$ at the vertex corresponding to 
the component of $\underline{\underline{m}}$ containing $\partial E$. 
In this way, each essential disk $E$ in 
$(E(\Gamma), \underline{\underline{m}})$ 
determine a Type I sphere $E \cup \hat{E}$ for $(S^3, \Gamma)$ up to isotopy. 
This correspondence gives 
an injection from the set of edges of $\mathcal{Z}$ 
and the set of isotopy classes of 
a pair of non-isotopic, weakly disjoint Type I spheres for $(S^3, \Gamma)$. 
Hence, the number of orbits of edges of $\mathcal{Z}$
by the action of $\MCG_+(E(\Gamma), \underline{\underline{m}})$ 
is equal to or less than 
the number of pairs of non-isotopic, weakly disjoint 
Type I spheres for $(S^3, \Gamma)$. 
Hence by Lemma \ref{lemma:finiteness of orbits}, the $2$-skeleton of 
$\mathcal{Z} / \MCG_+(E(\Gamma), \underline{\underline{m}})$ is finite. 

Let $E$ be a vertex of $\mathcal{Z}$. 
We proceed to show that $\MCG( E(\Gamma) , \underline{\underline{m}} , E)$ 
is finitely presented. 
%Then there exists an essential disk $\hat{E}$ in $E(\Gamma)$ 
%such that $\partial E = \partial \hat{E} $. 
Cutting $E(\Gamma)$ along $E$, we get two manifolds 
$L_1$ and $L_2$. 
Both $L_1$ and $L_2$ inherit boundary-patterns 
$(\underline{\underline{l_1}}, E^\prime)$ and 
$(\underline{\underline{l_2}} , E^{\prime \prime})$ with a spot 
from $\underline{\underline{m}}$ and $E$, respectively, 
as in the following way. 
In the case where $\partial E$ is isotopic in $\partial E(\Gamma)$ to 
the boundary of one, say $D_1^\prime$, of the separating disks 
$D_1^\prime, D_2^\prime , \ldots, D_{n_2}^\prime$, then 
we isotope $E$ so that $\partial E = \partial D_1^\prime$. 
Let $E^\prime$ and $E^{\prime \prime}$ be the disks in $\partial L_1$ and 
$\partial L_2$, respectively, that came from $E$. 
Let $\underline{\underline{l_i}}^\prime = \{ A \cap \partial L_i 
\mid A  \in \underline{\underline{m}} \} $ for $i=1,2$. 
Then there exists the unique element $A_0 \in L_1$ 
($B_0 \in L_2$, respectively) such that 
$\partial E^\prime \subset \partial A_0$
($\partial E^{\prime\prime} \subset \partial B_0$, respectively).  
Now, the boundary patterns $(\underline{\underline{l_1}}, E^\prime)$ 
and $(\underline{\underline{l_2}}, E^{\prime\prime})$ with a spot 
are defined by 
\[ (\underline{\underline{l_1}}, E^\prime) = 
(\underline{\underline{l_1}}^\prime 
\setminus \{A_0\} \cup \{A_0 \cup E^\prime\} , E^\prime), ~\mbox{ and }
(\underline{\underline{l_2}}, E^{\prime \prime}) = 
(\underline{\underline{l_2}}^\prime 
\setminus \{A_0\} \cup \{B_0 \cup E^{\prime \prime}\} , E^{\prime \prime}). \]
See Figure \ref{fig:cut_along_E}. 
\begin{figure}[htbp]
\begin{center}
\includegraphics[width=13cm,clip]{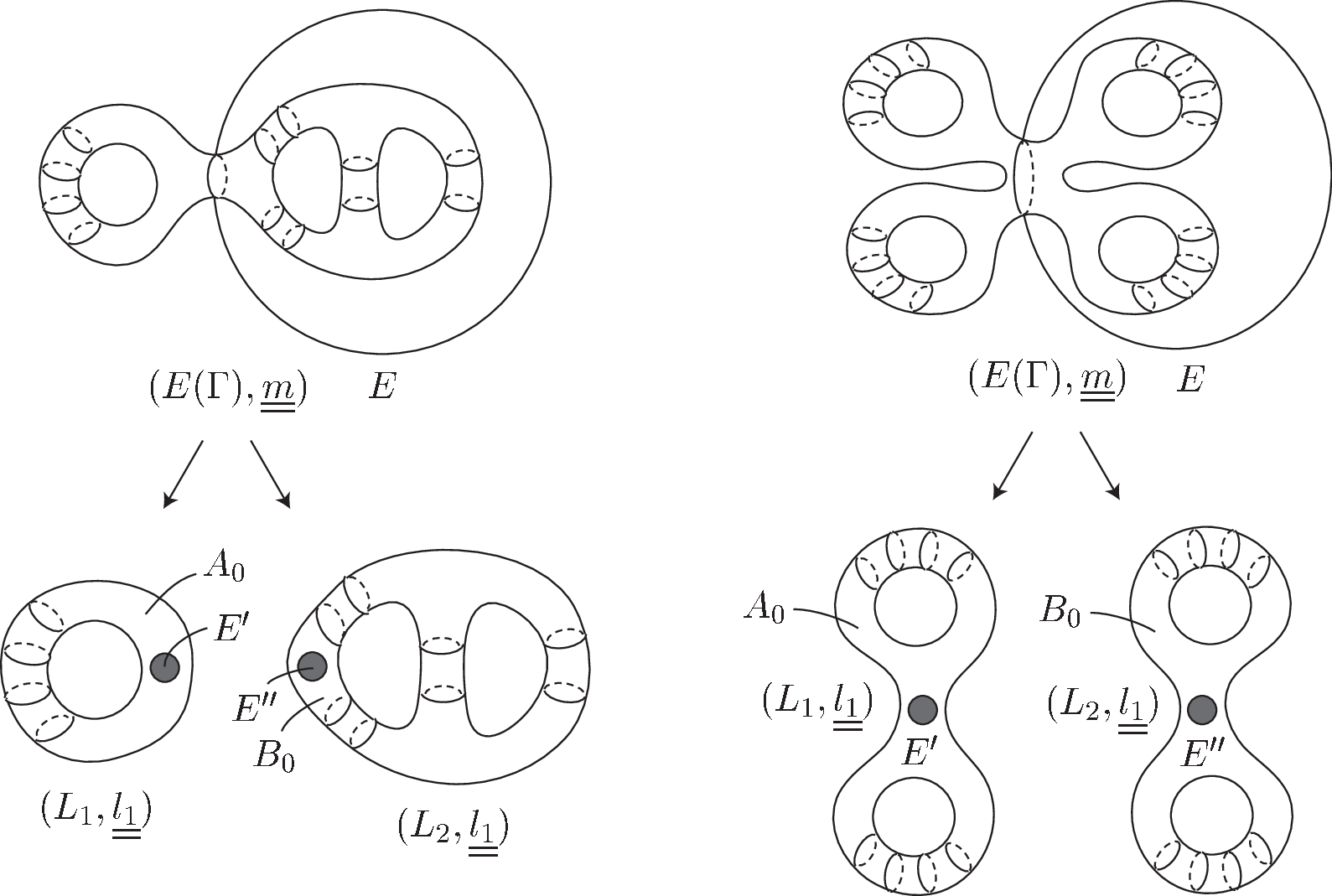}
\caption{From $(E(\Gamma), \underline{\underline{m}})$ and $E$ 
to $(L_1 , \underline{\underline{l_1}} , E^\prime)$ and 
 $(L_2 , \underline{\underline{l_2}} , E^{\prime \prime})$.}
\label{fig:cut_along_E}
\end{center}
\end{figure}
 
Note that $E^\prime$ and $E^{\prime \prime}$ are spots of 
$(L_1, \underline{\underline{l_1}})$ and 
$(L_2, \underline{\underline{l_2}})$, respectively. 
By the induction hypothesis and 
Lemmas 
\ref{lem:mapping class group of a 3-manifold with dotted boundary-pattern} 
and 
\ref{lem:mapping class group of a 3-manifold 
with dotted boundary-pattern 
whose dots are in a disk or an annulus}, 
it follows that 
both $\MCG (L_1, \underline{\underline{l_1}} ~\rel~ E^\prime)$ and 
$\MCG (L_2, \underline{\underline{l_2}} ~\rel~ E^{\prime \prime})$ 
are finitely presented. 

Let $\tau_1 \in \MCG_+ (L_1, \underline{\underline{l_1}} ~\rel~ E^\prime)$ and 
$\tau_2 \in \MCG_+ (L_2, \underline{\underline{l_2}} ~\rel~ E^{\prime \prime})$ 
be Dehn twists about $E^\prime$ and $E^{\prime \prime}$, respectively. 
Then we have 
\[ \MCG_+ (E(\Gamma) , \underline{\underline{m}} , E) 
\cong \MCG_+ (L_1, \underline{\underline{l_1}} ~\rel~ E^\prime) 
*_{\langle \tau_1 \tau_2^{-1} \rangle} 
\MCG_+ (L_2, \underline{\underline{l_2}} ~\rel~ E^{\prime \prime}) . \] 
%if there is no element in 
%$\MCG_+ (E(\Gamma) , \underline{\underline{m}}, E) $ 
%which restricts to an orientation reversing automorphism 
%of $E$ and 
%\[ \MCG_+ (E(\Gamma) , \underline{\underline{m}} , E) 
%\cong \MCG_+ (L_1, \underline{\underline{l_1}} ~\rel~ E^\prime) 
%*_{\langle \tau_1 \tau_2^{-1} \rangle} 
%\MCG_+ (L_2, \underline{\underline{l_2}} ~\rel~ E^{\prime \prime}) 
%\rtimes \Integer / 2 \Integer \] 
%otherwise. 
Now it follows from 
Lemma \ref{lem:exact sequence and finite presentation} that 
$\MCG(E(\Gamma) , \underline{\underline{m}} , E)$ is finitely presented. 

Let $\{ E_1 , E_2 \}$ be an edge of $\mathcal{Z}$. 
Then we can 
prove that $\MCG( E(\Gamma) , \underline{\underline{m}} , E_1 \cup E_2)$ 
is finitely generated by an analogous argument using 
Lemmas 
\ref{lem:mapping class group of a 3-manifold with dotted boundary-pattern} 
and 
\ref{lem:mapping class group of a 3-manifold 
with dotted boundary-pattern 
whose dots are in a disk or an annulus}  
as above. 

Consequently, the action $\MCG(E(\Gamma) , \underline{\underline{m}})$ 
on the contractible simplicial complex $\mathcal{Z}$ satisfies the assumption of 
Theorem \ref{thm:theorem by Brown} and thus we get 
that $\MCG(E(\Gamma) , \underline{\underline{m}})$ 
is finitely presented, which is our claim. 
\end{proof}

\begin{theorem}
\label{thm: symmetry group of a graph is finitely presented}
For every spatial graph $(S^3, \Gamma)$, 
the symmetry group 
$\MCG (S^3 , \Gamma)$ 
is finitely presented. 
\end{theorem}
\begin{proof}
It is clear that 
$\MCG(E(\Gamma), \underline{\underline{m}}) \cong \MCG (S^3 ~\rel~ \Gamma)$.  
By Lemma \ref{pro:symmetry groups and topological symmetry groups} 
there exists the following exact sequence$:$ 
\[ 1 \to \MCG (S^3 ~\rel~ \Gamma) \to \MCG (S^3 , \Gamma) \to 
\TSG (S^3 , \Gamma) \to 1 . \]
Since both $\MCG (S^3 ~\rel~ \Gamma)$ and 
$\TSG (S^3 , \Gamma)$ are finitely presented by 
Lemma \ref{lem:symmetry group of an irreducible 
3-manifold with useful boundary-pattern}, it follows from 
Lemma \ref{lem:exact sequence and finite presentation} that 
$\MCG (S^3 , \Gamma)$ is finitely presented. 
\end{proof}

\begin{example}
Figure \ref{fig:handcuff} illustrates 
a {\it handcuff graph} $\Gamma$ planarly embedded in $S^3$. 
\begin{figure}[htbp]
\begin{center}
\includegraphics[width=4cm,clip]{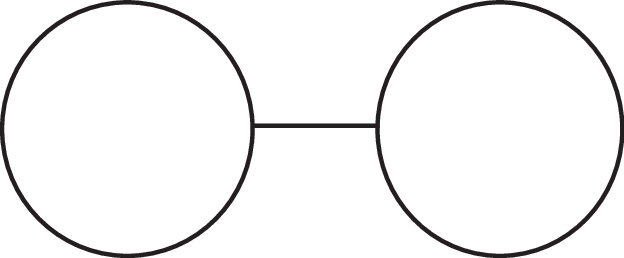}
\caption{A planarly embedded handcuff graph.}
\label{fig:handcuff}
\end{center}
\end{figure}
The graph $\Gamma$ has a unique cut-edge $e$ 
that is not a loop. 
This implies that all elements of $\MCG(S^3, \Gamma)$ preserve $e$. 
Thus we have 
\[ \MCG_+ (S^3 , \Gamma) = \langle g_1, g_2, g_3 \mid 
{g_1}^2, {g_3}^2, g_1 g_2 {g_1}^{-1} {g_2}^{-1}, g_1 g_3 {g_1}^{-1} {g_3}^{-1}, 
g_3 g_2 g_3 {g_2}^{-1} {g_1}^{-1}
 \rangle , \]
where $g_1, g_2, g_3$ are shown in Figure \ref{fig:handcuff_generators}. 
\begin{figure}[htbp]
\begin{center}
\includegraphics[width=12cm,clip]{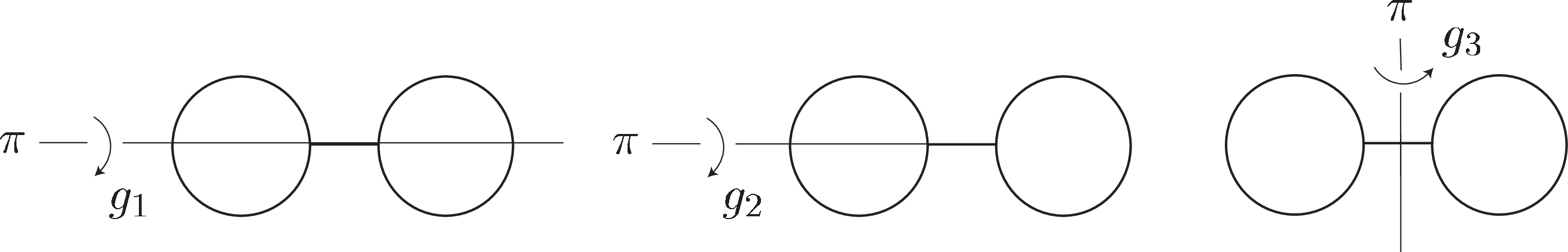}
\caption{Automorphisms $g_1$, $g_2$ and $g_3$.}
\label{fig:handcuff_generators}
\end{center}
\end{figure}

\end{example}

\section{The symmetry groups of reducible handlebody-knots of genus two} 
\label{sec:The symmetry groups of reducible handlebody-knots of genus two} 

Let $(S^3, V)$ be a handlebody-knot of genus two. 
In this section we assume that there exists 
a reducing sphere $P$ for $(S^3, V)$. 
The sphere $P$ separates $S^3$ into two 3-balls 
$B_1$ and $B_2$, i.e. $S^3 = B_1 \cup_{P} B_2$. 
Set $V_i = V \cap B_i$ for $i=1,2$ and 
let $K_1$ and $K_2$ be the cores of $V_1$ and $V_2$, respectively. 
Set $M_i = E(V) \cap B_i$ for $i=1,2$. 
We note that $M_i$ is homeomorphic to $E(K_i)$. 

\subsection{The case where both $K_1$ and $K_2$ are non-trivial}
\label{subsec:The case where both $K_1$ and $K_2$ are non-trivial}

In this subsection we always assume 
that both $K_1$ and $K_2$ are non-trivial knots. 
Figure \ref{fig:trefoil-figure_eight} illustrates 
a typical picture of $V$ in this case. 

\begin{figure}[htbp]
\begin{center}
\includegraphics[width=6cm,clip]{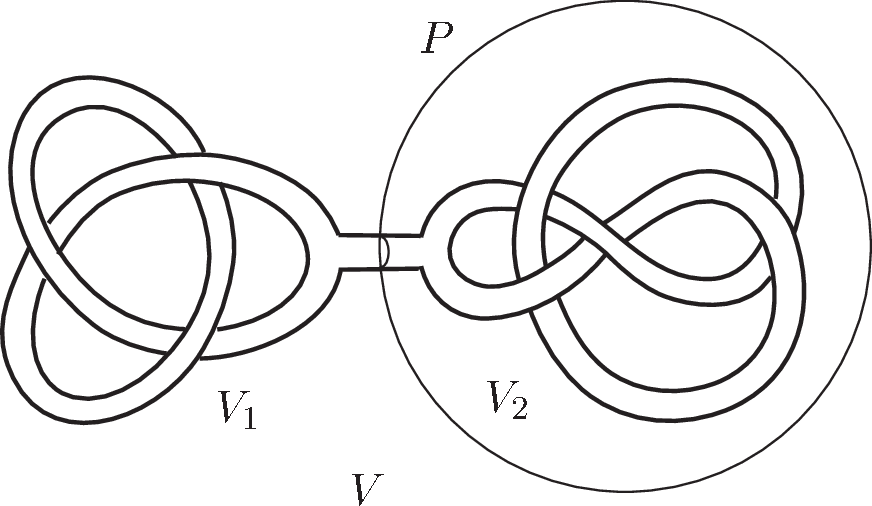}
\caption{}
\label{fig:trefoil-figure_eight}
\end{center}
\end{figure}

\begin{lemma}
\label{lem:uniqueness of P}
If both $K_1$ and $K_2$ are non-trivial, $E(V) \cap P$ is the 
unique $($up to ambient isotopy$)$ compression disk of $\partial E(V) $ in $E(V)$. 
\end{lemma}
\begin{proof}
Set $D = E(V) \cap P$. 
Let $D_1$ be an essential disk in $E( V )$. 
Isotope $D_1$ in order to minimize transverse intersections of 
$D$ and $D_1$ 
and suppose for a contradiction that 
$D \cap D_1 \neq \emptyset$. 
Since $E(V)$ is irreducible, there exists no loop in $D \cap D_1$. 
Let
$A \subset D_1$ be a disk cut off from $D_1$ by an outermost arc $\alpha$ 
of $D \cap D_1$ in $D_1$ such that
$A \cap D = \alpha$. 
We can assume that $A$ is contained in $M_1$. 
Then $A$ is a proper disk in $M_1 \cong E(K_1)$, 
and since $K_1$ is assumed to be 
non-trivial, $\partial A$ bounds 
a disk $A^\prime$ in $\partial M_1$. 
Since $E(V)$ is irreducible, $A \cup A^\prime$ 
bounds a 3-ball $B$ in $E(V)$. 
Therefore we can isotope $D_1$ in $N(B)$ 
so as to decrease $\#(D \cap D_1)$, which contradicts 
the minimality of the intersection. 
Thus $D \cap D_1 = \emptyset$. 
We can assume, without loss of generality, 
that $D_1 \subset M_1$. 
Since $K_1$ is non-trivial, 
$M_1 \cong E(K_1)$ is Haken. 
Hence $\partial D_1$ bounds a disk $A^{\prime \prime}$ in 
$\partial M_1$. 
By assumption, $\partial D_1$ is essential in $\partial E(V)$, thus 
essential in $\partial E(V) \cap \partial M_1$ as well. 
Thus $A^{\prime \prime}$ must contain 
$D \subset \partial M_1$. 
We conclude, by the irreducibility of 
$E(V)$, that $D_1$ is parallel to $D$. 
\end{proof}

We remark that from the Lemma \ref{lem:uniqueness of P} and 
irreducibility of the handlebody $V$, 
it follows that $P$ is the unique reducing sphere for $V$. 

\begin{lemma}
\label{lem:natural isomorphism for non-trivial-non-trivial case}
If both $K_1$ and $K_2$ are non-trivial, 
we have 
the natural isomorphism $\MCG_+(S^3, V) \cong \MCG_+(E(V))$. 
\end{lemma}
\begin{proof}
Let $\varphi : \MCG_+(S^3, V) \to \MCG_+(E(V))$ be the homomorphism 
that takes $f \in \MCG_+(S^3, V)$ to
$f|_{E(V)} \in \MCG_+ (E(V))$.  
By Lemma 
\ref{lem:injectivity of the map from MCG_+(S3, V) to MCG+(E(V))}, 
it remains to show the surjectivity of $\varphi$. 
Set $D = E(V) \cap P$. 
Take an arbitrary element $g \in \Aut_+(E(V))$. 
By Lemma \ref{lem:uniqueness of P}, we can isotope 
$g$ so that $g(D) = D$. 
It suffices to show that $g$ extends to an element of $\Aut_+(S^3, V)$.  

Suppose first that $g |_D$ is orientation-preserving. 
Then we have that $g |_{M_i} \in \Aut_+(M_i)$. 
Since $M_i \cong E(K_i)$, it follows from Gordon-Luecke \cite{GL89} that 
$g |_{M_i}$ extends to an element of $\Aut_+(B_i, V_i)$. 
Thus $g$ extends to an element of $\Aut_+(S^3, V)$ 
by Alexander's trick. 
The situation is only slightly more complicated 
if we drop the requirement that $g|_D$ 
be orientation-preserving. 
Next, suppose that $g |_D$ is orientation-reversing. 
We note that in this case $K_1$ and $K_2$ are the same knot. 
Then $g (M_i) = M_{i+1}$ (subscripts $\mathrm{mod}~ 2$). 
Again by \cite{GL89}, we see that 
$g$ takes a meridian slope of $V_i$ in 
$\partial M_i \setminus \Int \thinspace D$ to 
a meridian slope of $V_{i+1}$ in 
$\partial M_{i+1} \setminus \Int \thinspace D$. 
Therefore $g$ extends to an element of $\Aut_+(S^3, V)$.  
\end{proof}

The following proposition follows immediately 
from Lemmas \ref{lem:exact sequence and finite presentation} and 
\ref{lem:natural isomorphism for non-trivial-non-trivial case}, 
and 
Theorem \ref{thm:finitely presentability by McCullough}. 

\begin{proposition}
\label{prop:finite presentation of non-trivial non-trivial case}
If both $K_1$ and $K_2$ are non-trivial, 
then the group $\MCG(S^3, V)$ 
is finitely presented. 
\end{proposition}

We remark that by Lemma \ref{lem:uniqueness of P}, 
$\MCG(S^3, V)$ is actually nothing else but 
$\MCG (S^3, \Gamma)$, where 
$\Gamma$ is the spatial handcuff graph obtained by 
connecting $K_1$ and $K_2$ by a ``straight" line 
in $V$. 
From this viewpoint, Proposition 
\ref{prop:finite presentation of non-trivial non-trivial case}
can be regarded as a direct corollary of 
Theorem \ref{thm: symmetry group of a graph is finitely presented}.

\subsection{The case where $K_1$ is non-trivial and $K_2$ is trivial}
\label{subsec:The case where $K_1$ is non-trivial and $K_2$ is trivial}

In this subsection we always assume that 
only one of them (say $K_2$) is the unknot. 
(Hence $K_1$ is non-trivial.)  
Figure \ref{fig:trefoil-unknot} illustrates 
a typical picture of $V$ in this case. 

\begin{figure}[htbp]
\begin{center}
\includegraphics[width=6cm,clip]{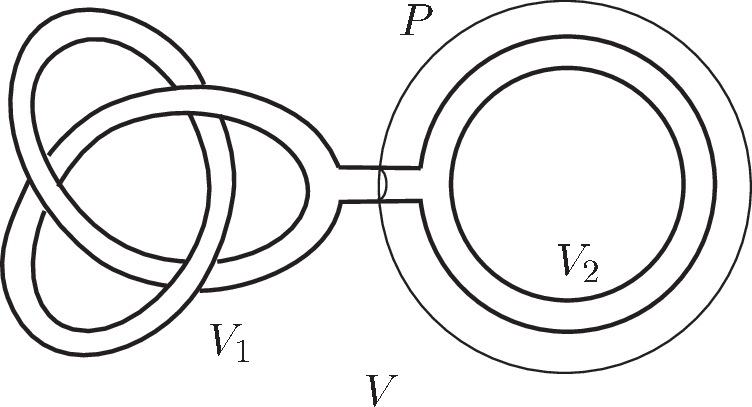}
\caption{}
\label{fig:trefoil-unknot}
\end{center}
\end{figure}

Let $W$ be a 3-manifold obtained 
from a closed, connected, orientable surface $F$ by 
attaching a single $2$-handle to 
$F \times \{ 1 \} \subset F  \times [0,1]$ 
along a non-separating simple closed curve $l$ 
in $F \times \{ 1 \}$. 
Let $\partial_+ W = F \times \{0\}$ and 
$\partial_- W = \partial W \setminus \partial_+ W$. 
A core disk of the 2-handle can
be vertically extended through $F \times [0,1]$ 
to a non-separating essential disk 
$E$ in $W$. 
The manifold $W$ is call a {\it compression body}. 
See e.g. Scharlemann \cite{Sch02} for general definition and 
properties of compression bodies. 

\begin{lemma}
\label{lem:non-sepa. disks of compression body 2}
Any non-separating essential disk in $W$ 
whose boundary lies in $\partial_+ W$ 
is parallel to $E$.  
\end{lemma}
\begin{proof}
Let $E^\prime$ be a non-separating essential disk in $W$ such that 
$\partial E^\prime \subset \partial_+ W$. 
Isotope $E^\prime$ in order to minimize $\#(E \cap E^\prime)$ 
and suppose for a contradiction that $E \cap E^\prime \neq \emptyset$. 
It is easy to see that $W$ is irreducible, 
hence all components of $E \cap E^\prime$ 
are arcs. 
Let $A \subset E^\prime$ be a disk cut off from 
$E^\prime$ by an outermost arc $\alpha$ 
of $E \cap E^\prime$ in $E^\prime$ such that $A \cap E = \alpha$. 
Then $A$ is a proper disk in the 
3-manifold $W^\prime$ obtained by cutting $W$ along 
$E$. 
Since $W^\prime$ is homeomorphic to $\partial_- W \times [0,1]$ and 
$\partial_- W \times [0,1]$ is boundary-irreducible, 
$\partial A$ bounds a disk $A^\prime$ in $\partial W^\prime$. 
Since $\partial_- W \times [0,1]$ is irreducible, $A \cup A^\prime$ bounds 
a 3-ball in $W^\prime$. 
We can isotope $E'$ so as to decrease $\# (E \cap E^\prime)$ 
, which contradicts the minimality of the intersection. 
Therefore $E$ is disjoint from $E'$. 
Now, the lemma follows from the fact that $W^\prime$ 
is irreducible and boundary-irreducible. 
\end{proof}

\begin{lemma}
\label{lem:E is unique}
Up to isotopy, there exists a unique non-separating disk in $E(V)$, 
which is prime. 
\end{lemma}
\begin{proof}
Set $W = M_2 \cup N(\partial M_1 ; M_1)$. 
We note that $W$ is a compression body 
satisfying the assumption of 
Lemma \ref{lem:non-sepa. disks of compression body 2}. 
In this case, we have 
$\partial_+ W = \partial E(V)$ and 
$\partial_- W = \partial W \setminus \partial_+ W$. 

The surface $\partial_- W$ separates 
$E(V)$ into two components 
$W$ and $E(W) = E(V) \setminus \Int \thinspace W$. 
Then the surface 
$\partial_- W$ is incompressible in $W$ since 
the negative boundary of a compression body 
is incompressible. 
The surface $\partial_- W$ is also 
incompressible in $E(W)$ since 
$(E(W) , \partial_- W)$ is homeomorphic to 
$(E(K) , \partial E(K))$. 
Thus $\partial_- W$ is incompressible in $E(V)$.  
Let $E$ be a non-separating disk in $E(V)$. 
Since $E(V)$ is irreducible and 
$\partial_- W$ is incompressible in $E(V)$, 
$E$ can be isotoped into $W$. 
It follows from Lemma \ref{lem:non-sepa. disks of compression body 2} 
that $E$ is the unique non-separating disk in $E(V)$. 
Since there exists a primitive disk in $E(V)$, 
$E$ is the unique primitive disk in $E(V)$. 
\end{proof}

\begin{lemma}
\label{lem:non-sepa. disk and loop}
Let $H$ be the genus two handlebody. 
Let $D_2 \subset H$ be an essential non-separating disk. 
Let $l$ be a non-separating simple closed curve in $\partial H$ 
such that $\partial D_2$ and $l$ have a single 
transverse intersection in $\partial H$. 
Let $D_1 \subset H$ be an essential non-separating disk such that 
$D_1 \cap (D_2 \cup l) = \emptyset$. 
If $D^\prime \subset H$ is an essential non-separating disk such that  
$\partial D^\prime$ and $l$ have a single 
transverse intersection in $\partial H$, 
$D^\prime$ can be isotoped in $H$ so that 
$D_1 \cap D^\prime = \emptyset$.  
\end{lemma}
\begin{proof}
If $D'$ intersects $D_1$, 
then there exist at least two outermost sub-disk, 
say $C_1$ and $C_2$, of $D'$ cut off by $D' \cap D_1$ such that 
$C_1$ and $C_2$ meet $D_1$ in opposite sides. 
Furthermore, each of $C_1$ and $C_2$ intersects 
$l$ since $l$ is a longitude of the solid torus $V$ cut off by $D_1$. 
This is a contradiction since $D'$ intersects $l$ 
in a single point. 
\end{proof}

The following lemma is an immediate consequence of 
Lemma \ref{lem:E is unique} and \ref{lem:non-sepa. disk and loop}. 
\begin{lemma}
\label{lem:D1 does not intersect primitive disks}
Let $D_1$ be a meridian disk of the solid torus 
$V_1$ such that $D_1$ is disjoint from 
the reducing sphere $P$. 
Then any primitive disk $D^\prime$ of $V$ 
can be isotoped in $V$ so that 
$D_1 \cap D^\prime = \emptyset$.  
\end{lemma}

The {\it primitive disk complex} $\mathcal{P}(V)$ of 
a handlebody-knot $(S^3, V)$ is defined to be the 
simplicial complex such that 
\begin{itemize}
\item
the set of vertices of $\mathcal{P}(V)$ 
consists of the ambient isotopy classes of primitive 
disks in $V$,
\item 
a collection of $k+1$ vertices spans a
$k$-simplex if and only if they admit a 
set of pairwise-disjoint representatives. 
\end{itemize}
We note that the primitive disk complex 
$\mathcal{P}(V)$ is a subcomplex of the 
{\it disk complex} of $V$. 

\begin{lemma}
\label{lem:contractibility of the primitive disk complex}
The primitive disk complex $\mathcal{P}(V)$ is contractible. 
\end{lemma}
\begin{proof}
Let $E$ be the unique primitive disk in $E(V)$. 
Let $D_2$ and $D_2^\prime$ be primitive disks in $V$. 
By Lemma \ref{lem:E is unique}, we can assume that 
$\partial D_2$ and $\partial E$ 
($\partial D_2^\prime$ and $\partial E$, respectively) 
have a single transverse intersection in $\partial V$, and that 
$\partial D_2 \cap \partial D_2^\prime \cap \partial E = \emptyset$. 
Further, assume that $D'_2$ intersects $D_2$ minimally. 
Assume that 
$D_2 \cap D_2^\prime \neq \emptyset$. 
Let $A \subset D_2$ be a disk cut off from $D_2$ by an outermost
arc $\alpha$ of $D_2 \cap D_2^\prime$ 
in $D_2$ such that $A \cap D_2^\prime = \emptyset$. 
The boundary $\partial \alpha$ of $\alpha$ separates 
$\partial D_2$ ($\partial D_2^\prime$, respectively)  
into two arcs $\delta_1$ and $\delta_2$ 
($\delta_1^\prime$ and $\delta_2^\prime$, respectively). 
Let $F_1$ and $F_2$ be disks
obtained from surgery on $D_2^\prime$ 
along $A$. 
We can assume that $\partial A = \delta_1 \cup \alpha$, 
$\delta_1^\prime \cap l \neq \emptyset$, and 
$\partial F_i = \delta_1 \cup \delta_i^\prime$ for $i=1,2$. 
If $\delta_1$ intersects $l$, 
the slopes $\partial F_2$ and $\partial E$ 
have a single transverse intersection in $\partial V$. 
If $\delta_1$ does not intersect $l$, 
the slopes $\partial F_1$ and $\partial E$ 
have a single transverse intersection in $\partial V$. 
Therefore one of $F_1$ and $F_2$ is primitive. 
Now the lemma follows from 
Theorem 4.2 in Cho \cite{Cho08}. 
\end{proof}

\begin{lemma}
\label{lem:dimension of the primitive disk complex}
The primitive disk complex $\mathcal{P}(V)$ is $1$-dimensional. 
\end{lemma}
\begin{proof}
It is clear that the dimension of $\mathcal{P}(V)$ is 
at least one. 
Let $D_2, D_2^\prime$ and $D_2^{\prime \prime}$ be 
pairwise disjoint primitive disks in $V$. 
By Lemma \ref{lem:D1 does not intersect primitive disks}, 
we can assume that the three disks do not 
intersect a meridian disk $D_1$ 
of the solid torus $V_1$ such that $D_1 \cap P = \emptyset$. 
Thus $D_2, D_2^\prime$ and $D_2^{\prime \prime}$ 
can be regarded as  
meridian disks of the solid torus 
$V^\prime = V \setminus \Int \thinspace 
N(D_1; V)$. 
Cut $V^\prime$ along 
$D_2 \cup D_2^\prime \cup D_2^{\prime \prime}$. 
Then $V^\prime$ decomposes into three components and 
at least one of them does not intersect $N(D_1; V)$. 
This implies that at least two of the three disks are 
parallel in $V$.  
\end{proof}

\begin{lemma}
\label{lem:cutting along a primitive disk}
The solid torus $V$ cut off by a primitive disk $D_2$ in $V$ 
is ambient isotopic to the solid torus $N(K_1)$. 
\end{lemma}
\begin{proof}
Let $E$ be the unique primitive disk in $E(V)$. 
Let $D_2$ be an arbitrary primitive disk in $V$. 
We can assume that 
$\partial D_2$ and $\partial E$ have a single 
transverse intersection. 
We note that $V \cup N( E ; E(V))$ is ambient isotopic 
to $N(K_1)$ since $V \cup N( E ; E(V))$ 
is obtained by attaching the 
3-ball $V_2 \cup N( E ; E(V))$ to $V_1 \cong N(K_1)$ along 
a disk $P \cap V$. 
The union $V \cup N( E ; E(V))$ is 
obtained from the solid torus 
$V \setminus \Int \thinspace N( D_2 ; V) $ 
by attaching a 3-ball $N( D_2 ; V) \cup N( E ; E(V)) $ 
along a disk, 
which is the union of the closure of 
$\partial N( D_2 ; V) \setminus \partial V$
and 
$N( E ; E(V)) \cap \partial V$. 
This implies 
that $V$ cut off by $D_2$ is isotopic to 
$N(K_1)$. 
\end{proof}

\begin{lemma}
\label{lem:stabilizer of primitive disks and pairs}
The group $\MCG_+(S^3, V)$ acts transitively on 
each of the sets of primitive disks of $V$ and primitive pairs of $V$. 
%The group $\MCG_+(S^3, V)$ acts transitively on 
%the set of primitive disks of $V$. 
%The group $\MCG_+(S^3, V)$ acts transitively on 
%the set of primitive pairs of $V$. 
\end{lemma}
\begin{proof}
Let $D_2$ and $D_2^\prime$ be primitive disks in $V$. 
By Lemma \ref{lem:cutting along a primitive disk}, 
we can regard $N(D_2; V)$ and $N(D_2^\prime; V)$ 
as regular neighborhoods of boundary-parallel proper 
arcs in $E(K_1)$. 
Since any two such arcs in $E(K_1)$ are ambient isotopic 
in $E(K_1)$, we get the first assertion. 

Let $\{ D_2 , D_2^\prime \}$ and 
$\{ D_3 , D_3^\prime \}$ be primitive pairs in $V$. 
We can assume that 
each of these primitive disks has 
a single transverse intersection with $\partial E$. 
Set $p_2 = \partial D_2 \cap E$ and 
$q_2 = \partial D_2^\prime \cap E$ 
($p_3 = \partial D_3 \cap E$ and 
$q_3 = \partial D_3^\prime \cap E$, respectively). 
Let $\alpha_2$ ($\alpha_3$, respectively) be a 
simple proper arc in $E$ connecting 
$p_2$ and $q_2$ ($p_3$ and $q_3$, respectively). 

Regard $N(E; E(V))$ as a product $E \times [0,1]$. 
Then the union 
of the arcs $(\partial D_2 \cup \partial D_2^\prime) \setminus 
(\partial E \times (0,1) )$ 
($(\partial D_3 \cup \partial D_3^\prime) \setminus 
(\partial E \times (0,1) )$, respectively) 
and the arcs $\alpha_2 \times \{ 0 , 1 \}$ 
($\alpha_3 \times \{ 0 , 1 \}$, respectively) 
is a simple closed curve in the boundary of the 
solid torus $V \cup (E \times [0,1]) \cong E(K_1)$ bounding 
a disk in  $V \cup (E \times [0,1])$. 
Since any simple closed curve in $\partial E(K_1)$ 
bounding a disk in $E(K_1)$ are ambient isotopic, 
we get the second assertion. 
\end{proof}

Let $\mathcal{P}^\prime(V)$ be the 
first barycentric subdivision of $\mathcal{P}(V)$. 
By Lemma \ref{lem:stabilizer of primitive disks and pairs}, 
the quotient of $\mathcal{P}^\prime(V)$ by the action of 
$\MCG_+(S^3, V)$ is a single edge. 
For convenience, we will not distinguish 
disks and homeomorphisms from their ambient isotopy
classes in our notations.
Fix two vertices $v_1 = {D_2}$ and $v_2 = \{ D_2 , D_2^\prime \} $ 
that are endpoints of an edge $e$ of $\mathcal{P}^\prime(V)$.
We note that 
$\Stab_{\MCG_+ (S^3, V)}(v_1) = \MCG_+ (S^3, V , D_2)$, 
$\Stab_{\MCG_+ (S^3, V)}(v_2) = \MCG_+ (S^3, V , D_2 \cup D_2^\prime )$ 
and  
$\Stab_{\MCG_+ (S^3, V)}(e) = \MCG_+ (S^3, V , D_2 , D_2^\prime)$.  

By the theory of groups acting on trees due to Bass and Serre 
\cite{Ser77}, 
the mapping class group $\MCG_+ (S^3, V)$ is a free product with 
amalgamation 
\[ \MCG_+ (S^3, V , D_2) *_{\MCG_+ (S^3, V , D_2 , D_2^\prime)} 
\MCG_+ (S^3, V , D_2 \cup D_2^\prime ) .\]

%Let $\mu$ be a simple closed curve 
%in $\partial E(K_1)$ bounding 
%a meridian disk in $N(K_1)$. 
%Set $F = N ( \mu ; \partial E(K_1)) $.  
%Divide $F$ into two disks $F_1$ and $F_2$ by 
%two disjoint essential arcs in $F$. 
%Set $F_3 = \partial E(K_1) \setminus \Int 
%\thinspace (F_1 \cup F_2)$. 
%Then the set 
%$\underline{\underline{m}} 
%= \{ F_1 , F_2 , F_3 \}$ 
%is a {\it useful boundary pattern}, 
%and hence  $\MCG_+ (E(K_1) ; \underline{\underline{m}}) 
%\cong \MCG_+ (E(K_1) ; F_1 , F_2 , F_3 )$ 
%is finitely presented due to Theorem 4.3.1 in \cite{McC91}. 

Regarding $N(E ; E(V))$ as a product $E \times [0,1]$, 
set $E_+ = E \times \{ 1 \}$ and  $E_- = E \times \{ 0 \}$. 
Set $\alpha = \partial D_2 \setminus (\partial E \times \partial I)$ and 
$\alpha^\prime = \partial D_2^\prime \setminus (\partial E \times \partial I)$. 
Recall that $V \cup (E \times [0,1])$ 
and $N(K_1)$ are ambient isotopic. 

%Let $D$ be a meridian disk of $N(K)$. 
%Set $F_1 = N(\partial D ; \partial N(K))$ and 
%$F_2 = \partial N(K) \setminus \Int F_1$. 
%Since $K$ is not trivial, $\underline{\underline{m}} = \{ F_1 , F_2 \}$ 
%is a complete, useful $\partial$-pattern of 
%$E(K)$. 
%Let $E_+$ and $E_-$ be spots of $(E(K) , \underline{\underline{m}} )$ 
%both lying in $F_1$. 
 
Let $p$ be a dot of $(E(K_1) , \partial N(K_1))$. 
By Proposition \ref{lem:isomorphism in the case of knots}, 
we have $\MCG(S^3 , N(K_1) , p) \cong \MCG(E(K_1) , p)$.  
Let $P : \pi_1(\partial N(K_1) , p) 
\to \MCG ( E(K_1), p )$ 
be the point-pushing map. 
The map $P$ induces a map 
$\bar{P} : \pi_1(\partial N(K_1) , p) 
/ \ker P 
\to \MCG ( E(K_1), p )$. 
Then there exists the Birman exact sequence 
\[ 1 \to \pi_1 ( \partial N(K_1) , p ) / 
 \ker P  \xrightarrow{\bar{P}} 
\MCG ( E(K_1), p ) \to 
\MCG ( E(K_1) )  \to 1 . 
\]
Since $\pi_1 ( \partial N(K_1) , p )$ is an abelian group, 
it is obvious that the group $\pi_1(\partial N(K_1) , p) 
/ \ker P$ is finitely presented. 
Hence $\MCG ( E(K_1), p )$ 
is finitely presented by Lemma 
\ref{lem:exact sequence and finite presentation}. 
%Then the set 
%$\underline{\underline{m}} 
%= \{ F_1 , F_2 , F_3 \}$ 
%is a useful boundary pattern, 
%and hence  $\MCG_+ (E(K_1) ; \underline{\underline{m}}) 
%\cong \MCG_+ (E(K_1) ; F_1 , F_2 , F_3 )$ 
%is finitely presented due to Theorem 4.3.1 in \cite{McC91}. 
%Since each automorphism of $E(K_1)$ takes a meridian slope 
%to a meridian slope, we have 
%the following isomorphisms of mapping class groups: 
%\begin{eqnarray*} 
%\MCG_+ (S^3, V , D_2) & \cong & 
%\MCG_+ (E(K_1) ; E_+ \cup E_-, B_1) \\ 
%& \cong & 
%\MCG_+ (E(K_1) , \underline{\underline{m}}) 
%\ltimes \Integer, 
%\end{eqnarray*} 
\begin{eqnarray*} 
\MCG_+ (S^3, V , D_2) & \cong & 
\MCG_+ (S^3, V \cup (E \times [0,1]) , E_+ \cup \alpha \cup E_-) \\ 
& \cong & 
\MCG_+ (S^3, N(K_1) , p) \rtimes \Integer \\ 
& \cong & 
\MCG_+ (E(K_1) , p) \rtimes \Integer , 
\end{eqnarray*} 
%i.e. $\MCG_+ (S^3, V , D_2) $ is 
%an extension of $\MCG_+ (E(K_1) , p)$ 
%by $\Integer$. 
where $\Integer$ is generated by an element of 

the Dehn twist about the disk 
$N(E_+ \cup \alpha \cup E_-; \partial (V \cup (E \times [0,1])))$.  
Recall the proof of 
Lemma \ref{lem:mapping class group of a 3-manifold with dotted boundary-pattern}. 
Therefore the group $\MCG_+ (S^3, V , D_2) $ 
is finitely presented by Lemma 
\ref{lem:exact sequence and finite presentation}. 

Let $D$ be a meridian disk of $N(K_1)$. 
Set $A = N( \partial D ; \partial N(K_1) )$. 
Assume that the dot $p$ is contained in $\Int \thinspace A$.  
Then we have 
\begin{eqnarray*} 
\MCG_+ (S^3, V , D_2 \cup D_2^\prime ) 
& \cong & \MCG_+ (S^3 , V \cup (E \times [0,1]) ,  
E_+ \cup \alpha \cup E_- \cup \alpha^\prime ) \\ 
& \cong & 
\MCG_+ (S^3, N(K_1) , A , p) 
\rtimes (\Integer / 2 \Integer) , 
\end{eqnarray*} 
where $\Integer / 2 \Integer$ is generated by 
an element of $\MCG_+ (S^3 , V \cup (E \times [0,1]) ,  
E_+ \cup \alpha \cup E_- \cup \alpha^\prime )$ 
that maps $E_+$ to $E_-$ and $\alpha$ to $\alpha'$. 
Finally, we have 
\begin{eqnarray*} 
\MCG_+ (S^3, V , D_2 , D_2^\prime) 
& \cong & \MCG_+ (S^3 , V \cup (E \times [0,1]) ,  
E_+ \cup \alpha \cup E_- \cup \alpha^\prime , \alpha) \\ 
& \cong & 
\MCG_+ (S^3, N(K_1) , A , p) . 
\end{eqnarray*} 
Hence $\MCG_+ (S^3, V , D_2 \cup D_2^\prime )$ and 
$\MCG_+ (S^3, V , D_2 , D_2^\prime) $ 
are finitely presented as well by 
Lemmas \ref{lem:exact sequence and finite presentation} 
and \ref{lem:mapping class group of a 3-manifold 
with dotted boundary-pattern 
whose dots are in a disk or an annulus}. 

Thus we obtain the following theorem.  

\begin{theorem}
\label{thm:finite presentation of non-trivial trivial case}
If $K_1$ is non-trivial and $K_2$ is trivial, 
then the group $\MCG(S^3, V)$ 
is finitely presented. 
\end{theorem}

\begin{example}
Consider the case where $K_1$ is a trefoil and 
$K_2$ is the unknot as was depicted in Figure \ref{fig:trefoil-unknot}. 
It is well known that 
$\MCG_+(S^3, N(K_1))$ is homomorphic to the 
group $\Integer / 2 \Integer$ 
generated by an automorphism $ f $ shown in Figure 
\ref{fig:trefoil_tube}, see e.g. Kawauchi \cite{Kaw96}. 
\begin{figure}[htbp]
\begin{center}
\includegraphics[width=4cm,clip]{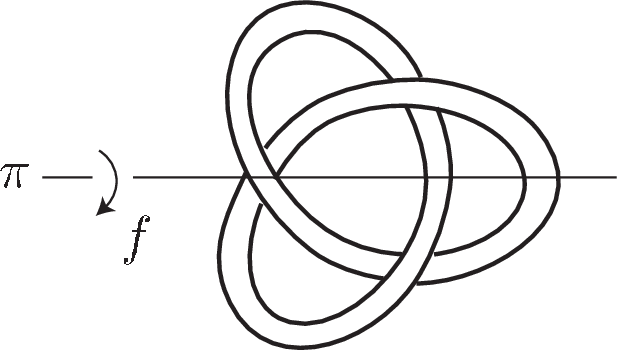}
\caption{The generator $f$ of $\MCG(S^3, N(K_1))$.}
\label{fig:trefoil_tube}
\end{center}
\end{figure}

\begin{figure}[htbp]
\begin{center}
\includegraphics[width=10cm,clip]{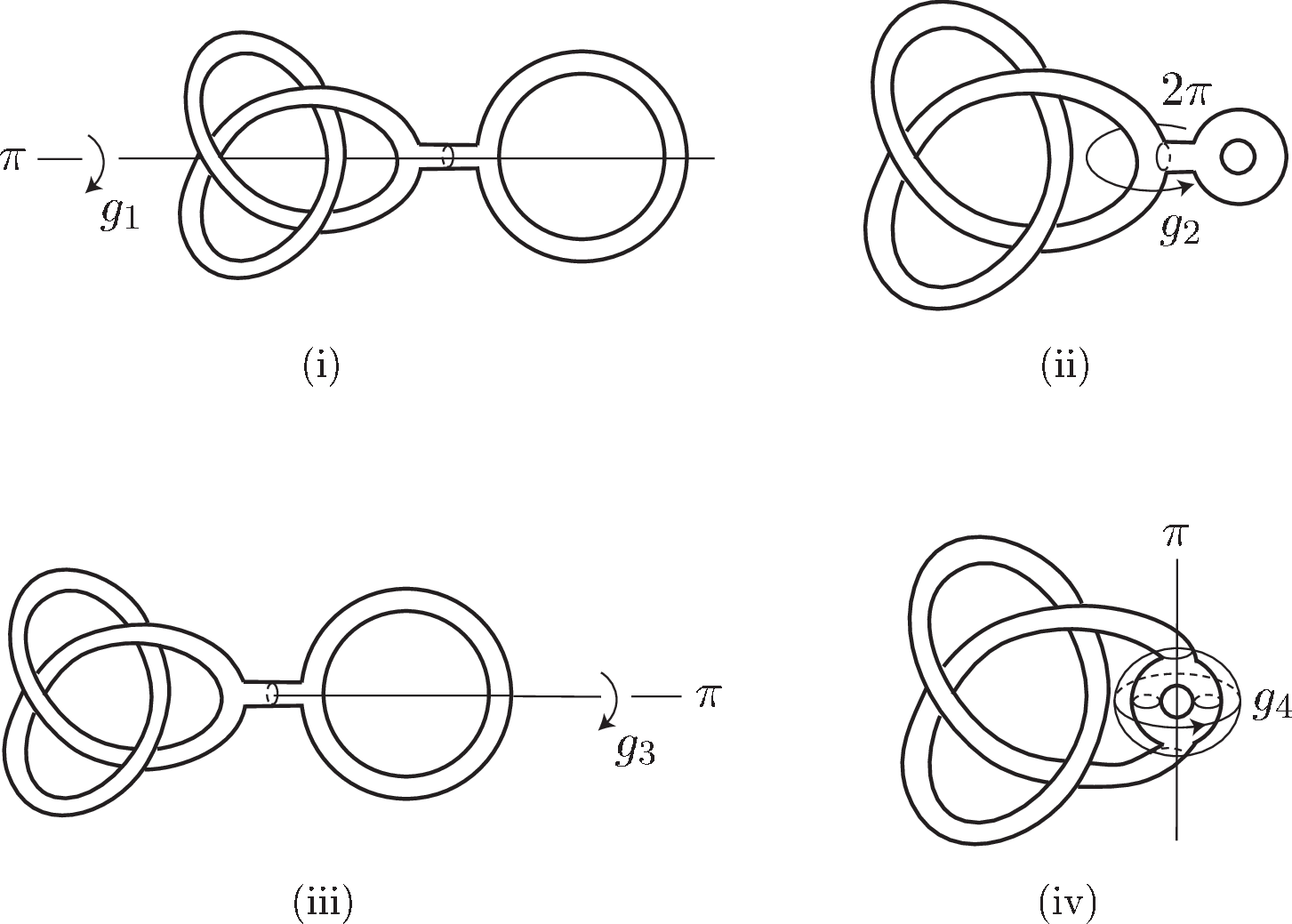}
\caption{Automorphisms $g_1 , g_2 , g_3$ and $g_4$.}
\label{fig:trefoil-unknot_generators}
\end{center}
\end{figure}
Let $g_1$, $g_2$, $g_3$, $g_4$ be elements of 
$\MCG(S^3, V)$ shown in Figure 
\ref{fig:trefoil-unknot_generators}. 
We note that $g_2 = {g_4}^2$. 
The element $g_1$ has order two while 
$g_3$ has infinite order. 
Since the trefoil is a fibered knot with the monodromy of order six, 
$g_2$ has order six.  

Let $D_2$ and $D_2^\prime$ be as in the above argument. 
We can compute the vertex and edge stabilizers of the tree  
$\mathcal{P}^\prime (V)$ as follows$:$   
\begin{eqnarray*}
\MCG_+(S^3, V, D_2) 
&=& \langle g_1, g_2, g_3 \mid 
{g_1}^2,  {g_2}^6 ,  
g_1 g_2 {g_1}^{-1} g_2 , 
g_1 g_3 {g_1}^{-1} {g_3}^{-1}, 
g_2 g_3 {g_2}^{-1} {g_3}^{-1}
\rangle \\
&\cong& ( (\Integer / 6 \Integer) 
\rtimes (\Integer / 2 \Integer)) \times \Integer , 
\end{eqnarray*}
\begin{eqnarray*}
\MCG_+(S^3, V, D_2 \cup D_2^\prime) &=& \langle g_1, g_2, g_4 \mid 
{g_1}^2,  {g_2}^6 , 
{g_4}^2 {g_2}^{-1} ,  g_1 g_4 {g_1}^{-1} g_4 \rangle \\
&\cong& (\Integer / 12 \Integer) \rtimes (\Integer / 2 \Integer) , 
\end{eqnarray*}
\begin{eqnarray*}
\MCG_+(S^3, V, D_2, D_2^\prime) = \langle g_1, g_2 \mid 
{g_1}^2,  {g_2}^6 , ~ 
g_1 g_2 {g_1}^{-1} g_2 \rangle 
\cong (\Integer / 6 \Integer) \rtimes (\Integer / 2 \Integer) . 
\end{eqnarray*}

Hence, we have 
\begin{eqnarray*}
\MCG_+(S^3, V) &=& 
\MCG_+ (S^3, V , D_2) *_{\MCG_+ (S^3, V , D_2 , D_2^\prime)} 
\MCG_+ (S^3, V , D_2 \cup D_2^\prime ) \\ 
&=& \langle g_1, g_3, g_4 \mid 	
{g_1}^2 , {g_4}^{12}, 
g_1 g_3 {g_1}^{-1} {g_3}^{-1}, 
g_1 g_4 {g_1}^{-1} g_4, 
g_3 g_4 {g_3}^{-1} {g_4}^{-1}
\rangle . 
%&&\MCG_+(S^3, V) \\
%&& {\hspace{1em}}
%= \langle g_1, g_3, g_4 \mid 	
%{g_1}^2 , {g_4}^{12}, 
%g_1 g_3 {g_1}^{-1} {g_3}^{-1}, 
%g_1 g_4 {g_1}^{-1} g_4, 
%g_2 g_3 {g_2}^{-1} {g_3}^{-1}
%\rangle 
\end{eqnarray*}

\end{example}

\vspace{1em}

Combining 
Theorems \ref{thm:genus two Goeritz group is finitely presented}, 
\ref{thm:finite presentation of non-trivial trivial case}
and Proposition \ref{prop:finite presentation of non-trivial non-trivial case}
we deduce the following conclusive theorem for the 
case of reducible handlebody-knots of genus two. 

\begin{theorem}
\label{thm:symmetry group of reducible handlebody-knots of genus two}
Let $V \subset S^3$ be a reducible handlebody-knot of genus two. 
Then the symmetry group $\MCG(S^3 , V)$ is 
finitely presented. 
\end{theorem}

\section{The symmetry groups of irreducible handlebody-knots of genus two} 
\label{sec:The symmetry groups of irreducible handlebody-knots of genus two}

A handlebody-knot of genus at least two is said to be {\it hyperbolic} 
if its exterior admits a complete hyperbolic structure with geodesic boundary. 
By Thurston's hyperbolization theorem \cite{Kap09} and 
equivariant torus theorems \cite{Hol91} 
a compact, orientable, irreducible, boundary-irreducible, 
atoroidal and acylindrical 3-manifold admits 
a hyperbolic structure with geodesic boundary.  
As was remarked in Section \ref{sec:Handlebody-knots}, 
if $(S^3, V)$ is an irreducible handlebody-knot of genus two, then 
$E(V)$ is boundary-irreducible. 
Hence we have the following lemma.   
\begin{lemma}
\label{lemma:hyperbolic handlebody-knots}
An irreducible handlebody-knot $(S^3, V)$ of genus two is hyperbolic if 
and only if $E(V)$ contains neither essential tori nor essential annuli.  
\end{lemma}

\begin{proposition}
\label{prop:symmetry groups of hyperbolic handlebody-knots}
If a handlebody-knot $(S^3, V)$ is hyperbolic, then 
$\MCG(S^3, V)$ is finite. 
\end{proposition} 
\begin{proof}
The mapping class group $\MCG_+(E(V))$ is finite 
due to the well-known fact that 
the isometry group of a compact hyperbolic manifold is finite, 
see Kojima \cite{Koj88}. 
It follows from 
Lemma \ref{lem:injectivity of the map from MCG_+(S3, V) to MCG+(E(V))} 
that $\MCG_+(S^3, V)$ is also finite, hence so is $\MCG(S^3, V)$. 
\end{proof}

Let $(S^3, V)$ be a handlebody-knot of genus two. 
In the following, we always assume that $(S^3, V)$ is irreducible. 
By Lemma \ref{lemma:hyperbolic handlebody-knots} and 
Proposition \ref{prop:symmetry groups of hyperbolic handlebody-knots}, 
it remains to consider the case where 
$E(V)$ contains an essential torus or an essential annulus. 
In the former case, $\MCG(S^3, V)$ is always an infinite group 
since the twist along an essential torus in $E(V)$ produces 
infinitely many distinct elements of $\MCG(E(V))$. 
On the contrary, we do not have any example of the latter case 
with $\MCG(S^3, V)$ an infinite group at the present time. 
In the following, we will provide a somewhat technical sufficient condition 
for the group $\MCG(S^3, V)$ to be a finite group.

Let $(S^3, V_0)$ be a trivial handlebody-knot of genus two. 
A non-separating simple closed curve on $\partial V_0$ 
is said to be {\it primitive} 
if it bounds a primitive disk in $V_0$. 
A separating simple closed curve on $\partial V_0$ 
is said to be {\it primitive} 
if it bounds disks in both $V_0$ and $E(V_0)$. 

\begin{lemma}
\label{lem:mcg of S3, an unknotting hdbdy and a proper annulus}
Let $(S^3, V_0)$ be a trivial handlebody-knot of genus two. 
Let $A_0$ be a non-separating annulus properly embedded in $V_0$. 
Set $\partial A_0 = C_0 \cup C_1$. 
Suppose that $A_0$ is compressible 
%both $C_0$ and $C_1$ are essential in $\partial A_0$ 
and that 
$C_0$ $($say$)$ is not primitive. 
Then $\MCG(S^3, V_0, A_0)$ is a finite group. 
\end{lemma}
\begin{proof}
We compress the annulus $A_0$ along a compression disk and 
denote the resulting disks by $D_0$ and $D_1$, 
where $\partial D_0 = C_0$ and $\partial D_1 = C_1$. 
Then it follows that $\MCG(S^3 , V_0 , A_0)$ is a subgroup of 
$\MCG(S^3 , V_0 , D_0 , D_1)$. 
Since $D_0$ is not primitive, the group $\MCG(S^3 , V_0 , D_0 , D_1)$ is 
finite by Cho-McCullough \cite[Proposition 17.2]{CM09}. 
\end{proof}

An annulus $A$ properly embedded in $E(V)$ is called an 
{\it unknotting annulus} of $(S^3, V)$ if 
$(S^3, V \cup N(A ; E(V)))$ is a trivial handlebody-knot of genus two. 

\begin{theorem}
\label{thm:unique unknotting annulus and symmetry groups}
Let $A$ be an unknotting annulus of $(S^3, V)$ 
and set $V_0 = V \cup N(A ; E(V))$. 
Suppose that $A$ is the unique unknotting annulus of $(S^3, V)$ 
up to ambient isotopy. 
Let $C$ be the core $S^1 \times \{ 1/2 \}$ of $A$.  
Regard $N( A ; E(V) )$ as $A \times [0,1]$ and 
set $A_0 = C \times [0,1] \subset N( A ; E(V) )$. 
%If $A_0$ is compressible and 
%two simple closed curves of $\partial A_0$ is not parallel in $\partial V_0$, 
If $A_0$ is compressible in $V_0$, 
then $\MCG(S^3 , V)$ is a finite group.  
\end{theorem}
\begin{proof}
It is clear from the definition that 
$A_0$ is non-separating in $V_0$. 
Set $\partial A_0 = C_0 \cup C_1$. 
Since $A_0$ is compressible, both $C_0$ and $C_1$ bound 
disks $D_0$ and $D_1$ in $V_0$, respectively. 
Then there exists a simple arc $l$ embedded in $V_0$ 
such that 
\begin{itemize}
\item
$l \cap (D_0 \cup D_1) = \partial l$,  
\item
$l \cap D_0 \neq \emptyset$ and $l \cap D_1 \neq \emptyset$, and 
\item
the closure of $N(l \cup D_0 \cup D_1 ; V_0) \setminus \partial V_0$ 
consists of three components $D_0^\prime$, $D_1^\prime$ and 
$A_0^\prime$, where 
$D_0^\prime$ is parallel to $D_0$, 
$D_1^\prime$ is parallel to $D_1$ and 
$A_0^\prime$ is parallel to $A_0$.  
\end{itemize}
See Figure \ref{fig:handlebody_and_annulus}. 
\begin{figure}[htbp]
\begin{center}
\includegraphics[width=8cm,clip]{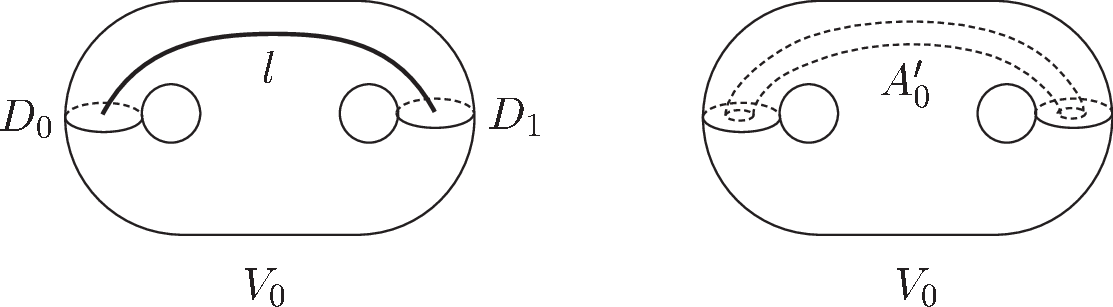}
\caption{Primitive disks $D_0$ and $D_1$ and an arc $l$ connecting them.}
\label{fig:handlebody_and_annulus}
\end{center}
\end{figure}
Since $A_0$ is non-separating, either $C_0$ or $C_1$ is non-separating. 
Assume that both $C_0$ and $C_1$ are primitive non-separating 
simple closed curves. 
Recall that the handlebody-knot $(S^3, V)$ is obtained by 
cutting $V_0$ along an annulus $A_0^\prime$. 
Then we can isotope $V_0$ in $S^3$ as shown in 
Figure \ref{fig:handlebody_and_annulus_deformation}. 
\begin{figure}[htbp]
\begin{center}
\includegraphics[width=14cm,clip]{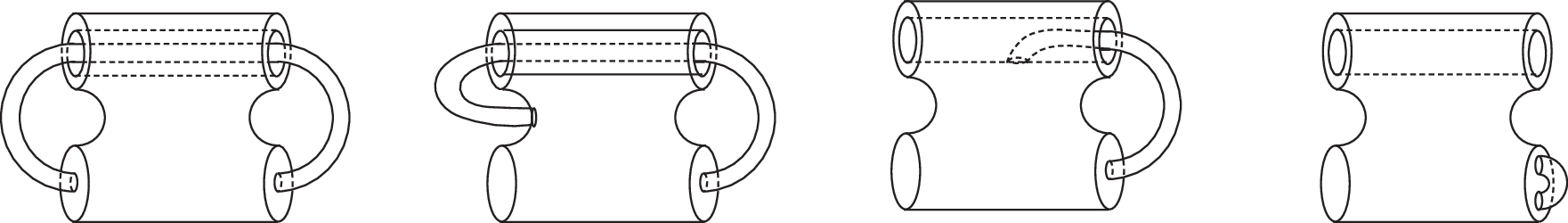}
\caption{Cutting $V_0$ along $A_0^\prime$ and an ambient isotopy.}
\label{fig:handlebody_and_annulus_deformation}
\end{center}
\end{figure}
The resulting figure implies that $(S^3, V)$ is reducible, which 
contradicts our assumption. 
(In fact, since $V$ is obtained from 
an unknotted solid torus by digging out a tunnel,  
it follows from Gordon \cite{Gor87} that 
$(S^3, V)$ is an unknotted handlebody of genus two.)
We can also deduce analogously a contradiction in the case 
where $C_0$ is a primitive separating simple closed curve and 
$C_1$ is a primitive non-separating one.  
As a consequence, $A_0$ satisfies the 
assumption of 
Lemma \ref{lem:mcg of S3, an unknotting hdbdy and a proper annulus}. 

Let $f$ be an element of $\MCG(S^3 , V)$. 
Since $A$ is the unique unknotting annulus for $(S^3, V)$, 
we can assume that $f (N(A ; E(V))) = N(A ; E(V))$. 
%Define a map $h \in \Aut (E (V) , N(A ; E(V)))$ so that 
%$h |_{ E(N(A ; E(V))) } = \id$ and 
%$h |_{N(A ; E(V)) ((S^1 \times I ) \times I)}$ maps 
%$((x,y), t)$ to $(x + 2 \pi t , y) , t)$. 
%Let $r$ be an element of $\Aut (E (V) ; N(A , E(V)))$ 
%(if any) such that 
%$r (A \times \{ 0 \}) = A \times \{ 1 \}$ and 
%$r (A \times \{ 1 \}) = A \times \{ 0 \}$. 
%Let $\tilde{f}$ be an element of $\Aut(S^3, V, N(A ; E(V)))$ (if any) 
%such that $\tilde{f} |_{E(V)} = f$ and 
%$\tilde{r}$ be an element of $\Aut(S^3, V, N(A ; E(V)))$ (if any) 
%such that $\tilde{r} |_{E(V)} = r$. 
%Let $H$ be a subgroup of $\MCG(S^3, V)$ generated by 
%$\tilde{h}$, which is infinite cyclic. 
%Then it follows that 
Hence $f$ uniquely determines an element $\tilde{f}$ of 
$\MCG(S^3 , V_0 , A)$. 
Conversely, each element of $\MCG(S^3 , V_0 , A)$ 
uniquely determines an element of $\MCG(S^3 , V)$. 
Thus we have  
\[
\MCG(S^3 , V) \cong  \MCG(S^3, V_0, A_0 )
\] 
and the  assertion follows immediately from  
Lemma \ref{lem:mcg of S3, an unknotting hdbdy and a proper annulus}. 
\end{proof}

%The Svar\'c-Milnor-Lemma and 
%the proof of 
%Theorem \ref{thm:unique unknotting annulus and symmetry groups} 
%allow us to deduce the following corollary: 
%
%\begin{corollary}
%\label{cor:finite unknotting annuli and symmetry groups}
%Suppose that $V$ admits 
%only finitely many unknotting annuli $A_1, A_2 , \ldots , A_n$ 
%up to ambient isotopy. 
%Let $C^i$ be the core of $A_i$.  
%Regard $N( A_i ; E(V) )$ as $A_i \times [0,1]$ and 
%set $A_i^d = C^i \times [0,1] \subset N( A_i ; E(V) )$. 
%If $A_i^d$ is compressible in $V \cup N(A_i ; E(V))$ for 
%each $i=1, 2, \ldots , n$, 
%then $\MCG(S^3 , V)$ is a finite group.  
%\end{corollary}

\begin{example}
The left-hand side in Figure \ref{fig:irreducible_handlebody-knots_and_annulus} 
illustrates a series of handlebody-knots $(S^3, V_n)$ $(n \geqslant 3)$ of genus two. 
We note that in the table of irreducible handlebody-knots in Ishii-Kishimoto-Moriuchi-Suzuki \cite{IKMS12}, 
$(S^3, V_3) = 5_1$ and $(S^3, V_4) = 6_1$. 
It follows easily from the arguments in Motto \cite{Mot90} and Lee-Lee \cite{LL12} 
that these handlebody-knots are irreducible and 
that the annulus $A$ depicted in the figure is a 
unique unknotting annulus properly embedded in 
$E(V_n)$. 
Moreover, it follows that 
all of these handlebody-knots are not {\it amphicheiral}, i.e. 
$\MCG(S^3, V_n) \cong \MCG_+(S^3, V_n)$.  
\begin{figure}[htbp]
\begin{center}
\includegraphics[width=10cm,clip]{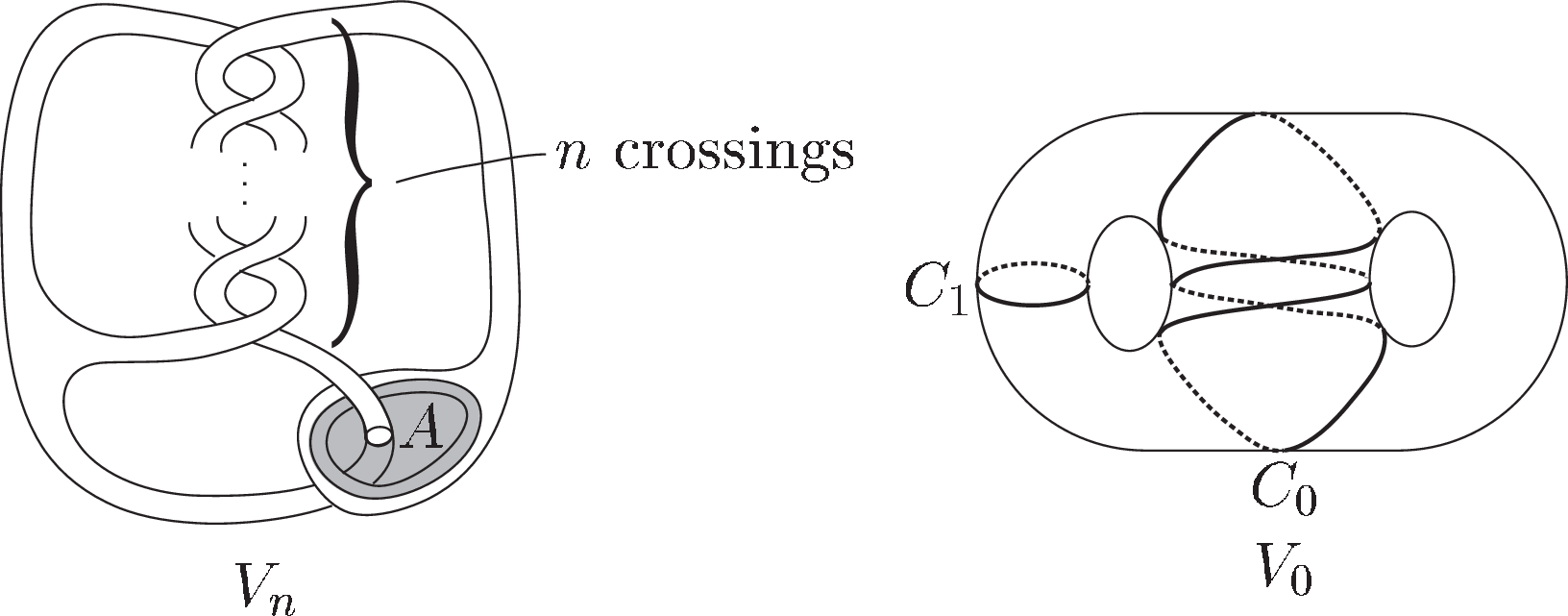}
\caption{The irreducible handlebody-knot $(S^3, V_n)$ and the curves $C_0$ and $C_1$.}
\label{fig:irreducible_handlebody-knots_and_annulus}
\end{center}
\end{figure}
The right-hand side in the figure shows the curves $C_0$ and $C_1$ in the case $n=3$. 
The ``hyperelliptic" involution is the unique non-trivial element of 
$\MCG ( S^3, V_0 )$  
which preserve the curves $C_0 \cup C_1$. 
However, this involution does not preserve the annulus $A_0$, 
hence we have $\MCG(S^3, V_n) = 1$. 
\end{example}

\noindent {\bf Acknowledgments.} 
The author wishes to express his gratitude 
to Sangbum Cho for his helpful comments 
on a preliminary version 
of the paper.
Part of this work was carried out while the author was visiting 
the Mathematisches Forschungsinstitut Oberwolfach. 
The institute kindly offered an invitation for the stay 
while his University was affected 
by the 2011 Tohoku earthquake. 
He is grateful to the MFO and its staffs for 
the offer, the financial support and warm hospitality. 
The author wrote the paper in the current form form during his visit to 
Universit\`a di Pisa as a 
JSPS Postdoctoral Fellow for Reserch Abroad. 
He is grateful to the university and its staffs for 
the warm hospitality. 
Finally, he is grateful to the anonymous referee 
for his or her comments that improved the exposition.

\end{document}